\documentclass[12pt]{article}
\pdfoutput=1
\usepackage{amsmath,amssymb,amsthm,tikz}
\usetikzlibrary{decorations.markings}
\usepackage{comment}
\usepackage{hyperref}
\usepackage{stix}

\newtheorem{theorem}{Theorem}

\newtheorem{lemma}[theorem]{Lemma}
\newtheorem{proposition}[theorem]{Proposition}

\theoremstyle{definition}
\newtheorem{definition}[theorem]{Definition}
\newtheorem{remark}[theorem]{Remark}

\numberwithin{equation}{section}
\numberwithin{theorem}{section}

\DeclareMathOperator{\Tr}{Tr}
\newcommand{\ds}{\displaystyle}

\let\Im\undefined

\DeclareMathOperator{\Im}{Im}

\title{Wiener--Hopf factorizations and matrix-valued orthogonal polynomials}

\author{Arno B.J. Kuijlaars and Mateusz Piorkowski\footnote{
Department of Mathematics, Celestijnenlaan 200 B, Box 2400, KU Leuven, Belgium. \\
Email addresses: \texttt{arno.kuijlaars@kuleuven.be} and \texttt{mateusz.piorkowski@kuleuven.be}}}

\begin{document}

\maketitle


\begin{abstract}
	We compare two methods for analysing periodic dimer models. These are the matrix-valued orthogonal polynomials approach due to Duits and one of the authors, and the Wiener--Hopf approach due to Berggren and Duits. We establish their equivalence in the special case of the Aztec diamond. Additionally, we provide explicit formulas for the matrix-valued orthogonal polynomials/Wiener--Hopf factors in the case of the $2 \times 2$-periodic Aztec diamond in terms of Jacobi theta functions related to the spectral curve of the model.  
\end{abstract}

\section{Introduction}
\subsection{Motivation}
Weighted models of the Aztec diamond, in particular those with a doubly periodic weight structure, have generated considerable interest in recent years. One of the driving forces behind this development has been the seminal paper of Kenyon, Okounkov and Sheffield \cite{KOS06}. There, the authors provide a description of possible Gibbs measures that can appear in general dimer models with doubly periodic weights. These Gibbs measures can be classified into three types usually referred to as \emph{smooth} (also gaseous), \emph{rough} (also liquid) and \emph{frozen} (also solid). Characteristic for each class is the decay rate of dimer correlations; exponential decay in smooth phases, quadratic decay in rough phases and no decay (i.e., perfect correlation) in frozen phases. It turns out that the space of Gibbs measures is parametrized by the amoeba of a certain spectral curve which has the Harnack property.

It is natural to expect that for large tilings of a finite domain the same type of Gibbs measures would appear locally. Significant progress on this problem for various tiling models has been achieved in the past decade, see e.g.~\cite{B21, BB23+, BD19, CDKL20, CJ16, CY14, DK21, DiFS14}. Such models also give rise to the appearance of arctic curves separating the rough region from the frozen and smooth regions, and the dimer statistics in the vicinity of such curves has generated considerable interest, see e.g.~\cite{BCJ18, BCJ22, CJ16, Jo05, JM23}. 

Due to a result of Kenyon \cite{K97}, it is known that the location of dimers in a dimer configuration gives rise to a determinantal point process. Here, the correlation kernel is essentially given through the inverse of the Kasteleyn matrix. Thus, one way to characterize the three regions mentioned above is via the asymptotic analysis of said correlation kernel. To this end, it is necessary to express it in a form suitable for asymptotic analysis. Here we mention two techniques which were used successfully for this purpose: the inversion of the Kasteleyn matrix via a recursion based on the domino shuffle \cite{CY14}, and the matrix-valued orthogonal polynomials (MVOPs)/Wiener--Hopf method \cite{BD19, DK21}. While the first method uses Kasteleyn theory, the second method uses a bijection between tilings of the Aztec diamond and  non-intersecting paths, together with the Eynard--Mehta Theorem \cite{EM98}.

In the first part of the present paper we take a closer look at the relation between the MVOPs method \cite{DK21} and the Wiener--Hopf method \cite{BD19}. Our results contained in Sections~\ref{MVOPs_WH}--\ref{SectkPer} imply that both approaches are equivalent in the case of the periodic Aztec diamond defined below. We also consider general matrix-valued contour orthogonality with respect to a rational matrix-valued weight, and show under which conditions certain MVOPs can be constructed in terms of a Wiener--Hopf factorization of the weight matrix. That the Wiener--Hopf construction of the MVOPs is indeed possible in the case of the periodic Aztec diamond follows from the results of \cite{BD19}, see also Section~\ref{SectkPer}. This statement can be viewed as a natural continuation of the recent paper \cite{CD23}, where the equivalence between the domino shuffle and Wiener--Hopf factorizations was established. 

While these results are known to experts \cite{BB23+}, they have never been written down in detail. Interestingly, they also imply an extension of a double integral formula due to Berggren and Duits \cite{BD19}, which is stated in Eq.~\eqref{DIntPhi}. The equivalence between MVOPs and Wiener--Hopf factorization does not appear to extend to tilings of the hexagon, though we will not address this topic here, cf.~\cite{CDKL20}. 

The second part of our paper deals with explicit formulas for the MVOPs (or equivalently Wiener--Hopf factors). In has been pointed out in \cite[p.~59]{BB23+}, \cite[p.~39]{BD19} that the iterative computation of the Wiener--Hopf factors according to \cite{BD19} leads to a (seemingly) intractable non-linear recursion, which was later identified to be the domino shuffle in \cite{CD23}. Instead of solving this recursion, we present in Section~\ref{2x2Section} an alternative approach based on a linear flow on the spectral curve. A special case of this construction (though not in terms of Jacobi theta functions) was recently done in \cite{BD23}, where the aformentioned linear flow was introduced, see also \cite{MV91}. Our approach allows us to obtain explicit formulas for the MVOPs (or Wiener--Hopf factors) for \emph{generic} $2\times2$-periodic models of the Aztec diamond. The key step in our proof is the diagonalization of the weight matrix and the factorization of the associated eigenvalue, viewed as a meromorphic function on the spectral curve, into a product of two multivalued functions. The multivaluedness is the main novelty compared to \cite{BD23} and leads naturally to Jacobi theta functions. Note that due to our $2\times2$-periodicity assumption, the genus of the spectral curve is one. 

As the underlying constructions can be performed explicitly in terms of Jacobi theta functions, the formulas for the MVOPs in terms of these special functions follow, see Theorem \ref{theo21}. Showing this is technically the most demanding part of the paper and topic of Section~\ref{Sect2x2Const}. However, we are optimistic that with more effort an analog of Theorem \ref{theo21} in the case of higher periodicity (hence higher genus) can be shown using similar techniques.

\subsection{The weighted Aztec diamond} Let us now set the stage by introducing the weighted Aztec diamond. We will keep our exposition brief, as this model has been described in various works in great detail, see e.g.~\cite{BB23+, DK21}. 

\subsubsection{Dimer models}
The Aztec diamond is a certain diamond shaped region in the plane, consisting of translates of unit squares. The task is to tile this region with horizontal and vertical dominos, see Figure~\ref{AztecDiamond} on the right for an example of a tiling of an Aztec diamond of size $4$. Any such tiling of the Aztec diamond is in bijection to a perfect matching (or dimer configuration) of a certain graph that we will define in a moment, see Figure~\ref{AztecDiamond} on the left. In the following, we will use this equivalence to define weighted models of the Aztec diamond.

\begin{figure}[h]
    \centering
    \begin{tikzpicture}[scale=0.6]

\foreach \j in {0,...,3}{
    \foreach \k in {0,...,4}{
        \node
        (\j, \k) at (0.5-4.0+\j+\k,-0.5-\j+\k) {$\bullet$};
    }
}

\foreach \j in {0,...,4}{
    \foreach \k in {0,...,3}{
        \node
        (\j, \k) at (0.5-4.0+\j+\k,0.5-\j+\k) {$\circ$};
    }
}
\foreach \j in {0,...,3}{
    \foreach \k in {0,...,3}{
        \draw (0.5-4.0+\j+\k,-0.5-\j+\k) -- (0.5-4.0+\j+\k+0.93,-0.5-\j+\k);
        \draw (0.5-4.0+\j+\k,-0.5-\j+\k) -- (0.5-4.0+\j+\k,-0.5-\j+\k+0.93);
    }
}

\foreach \j in {0,...,3}{
    \foreach \k in {1,...,4}{
        \draw (0.5-4.0+\j+\k,-0.5-\j+\k) -- (0.5-4.0+\j+\k-0.93,-0.5-\j+\k);
        \draw (0.5-4.0+\j+\k,-0.5-\j+\k) -- (0.5-4.0+\j+\k,-0.5-\j+\k-0.93);
    }
}
	    \draw [line width = 1mm] (-0.4,3.5)--(0.4,3.5);
	    \draw [line width = 1mm] (0.6,2.5)--(1.4,2.5);
	    \draw [line width = 1mm] (-0.5,2.4)--(-0.5,1.6);
	    \draw [line width = 1mm] (-1.5,2.4)--(-1.5,1.6);
	    \draw [line width = 1mm] (-2.5,1.4)--(-2.5,0.6);
	    \draw [line width = 1mm] (0.5,1.4)--(0.5,0.6);
	    \draw [line width = 1mm] (1.5,1.4)--(1.5,0.6);
	    \draw [line width = 1mm] (2.5,1.4)--(2.5,0.6);
	    \draw [line width = 1mm] (-3.5,0.4)--(-3.5,-0.4);
	    \draw [line width = 1mm] (-1.4,0.5)--(-0.6,0.5);
	    \draw [line width = 1mm] (3.5,0.4)--(3.5,-0.4);
	    
	    \draw [line width = 1mm] (-2.5,-0.6)--(-2.5,-1.4);
	    \draw [line width = 1mm] (-1.5,-0.6)--(-1.5,-1.4);
	    \draw [line width = 1mm] (-0.5,-0.6)--(-0.5,-1.4);
	    \draw [line width = 1mm] (0.6,-0.5)--(1.4,-0.5);	    
	    \draw [line width = 1mm] (2.5,-0.6)--(2.5,-1.4);
	    
	    \draw [line width = 1mm] (0.5,-1.6)--(0.5,-2.4);
	    \draw [line width = 1mm] (1.5,-1.6)--(1.5,-2.4);
	    \draw [line width = 1mm] (-1.4,-2.5)--(-0.6,-2.5);
	    \draw [line width = 1mm] (-0.4,-3.5)--(0.4,-3.5);

				
				\foreach \x/\y in {7/0,8/1,8/-1,9/2,10/-1,11/-2,12/1}
				{ \draw
					(\x,\y)--(\x,\y-1)--(\x+1,\y-1)--(\x+1,\y); 
					\draw [line width = 1mm] (\x,\y-1) rectangle (\x+1,\y+1);
				}
				
				\foreach \x/\y in {9/-1,10/2,11/1,12/-2,13/-1,13/1,14/0}
				{  \draw
					(\x,\y)--(\x,\y+1)--(\x+1,\y+1)--(\x+1,\y);
					\draw [line width = 1mm] (\x,\y-1) rectangle (\x+1,\y+1);
				}
				
				\foreach \x/\y in {11/-3,10/-2,12/0}
				{  \draw
					(\x,\y-1)--(\x-1,\y-1)--(\x-1,\y)--(\x,\y);
					\draw [line width = 1mm] (\x-1,\y-1) rectangle (\x+1,\y);
				}
				
				\foreach \x/\y in {10/1,11/4,12/3}
				{  \draw
					(\x,\y-1)--(\x+1,\y-1)--(\x+1,\y)--(\x,\y);
					\draw [line width=1mm] (\x-1,\y-1) rectangle (\x+1,\y);
				}
				
    \end{tikzpicture}
    \caption{\label{AztecDiamond} A dimer configuration and the corresponding tiling of the Aztec diamond of size $4$.}
    
\end{figure}
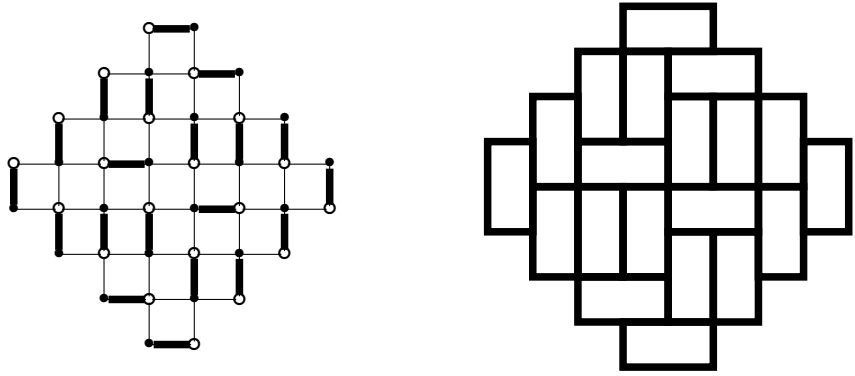

A dimer model is a probability measure on perfect matchings of a graph $\mathcal G$. Consider the graph $\mathcal G = (\mathcal V, \mathcal E)$, where $\mathcal V$ is the set of vertices and $\mathcal E$ is the set of edges. A perfect matching $M \subseteq \mathcal E$ of the graph $\mathcal G$ is a subset of edges such that each vertex of $\mathcal V$ is covered by exactly one edge in $M$. An edge $e$ of a perfect matching $M$ is called a \emph{dimer}, while $M$ is called a \emph{dimer configuration}. We denote by $\mathcal M$ the set of dimer configurations of a graph $\mathcal G$. A dimer model is defined by a probability measure 
\begin{align*}
\text{Prob} \colon \mathcal M  \to \mathbb{R}_+
\end{align*}
on the set of dimer configurations. A common choice for $\text{Prob}$ is
\begin{align}\label{ProbMeas}
    \text{Prob}(M) = \frac{\prod_{e \in M} w(e)}{\sum_{M'\in \mathcal M} \prod_{e' \in M'} w(e')},
\end{align}
where $w \colon \mathcal E \to \mathbb R_+$ is a given weight function on the edges. Here we implicitly assume that at least one dimer configuration of $\mathcal G$ exists.  

For the special case of the Aztec diamond of size $N$, the graph $\mathcal G_N = (\mathcal V_N, \mathcal E_N)$ is bipartite 
\begin{align*}
    \mathcal V_N = \mathcal B_N \cup \mathcal W_N,
\end{align*}
with black and white vertices  contained in
\begin{align*}
	\mathcal B_N &=  \big\lbrace (\tfrac{1}{2}-N+j+k, \, -\tfrac{1}{2}-j+k) \mid j = 0, \dots, N-1, \ k = 0, \dots, N \big\rbrace,
	\\
	\mathcal W_N &=  \big\lbrace (\tfrac{1}{2}-N+j+k, \, \tfrac{1}{2}-j+k) \mid j = 0, \dots, N, \ k = 0, \dots, N-1 \big\rbrace,
\end{align*}
as shown in the left panel of Figure~\ref{AztecDiamond}. 
The edge set $\mathcal E_N$ connects nearest neighbours only. 

\subsubsection{Periodic models of the Aztec diamond} 
To parametrize weighted models of the Aztec diamond, let us fix a positive integer $k \in\mathbb N$ and consider the Aztec diamond of size $kN$. We then rotate the graph $\mathcal G_{kN}$ and associate to it the set $\mathcal E_{kN}$ of edge weights given by $\alpha_{j,i}$, $\beta_{j,i}$, $\gamma_{j,i} >0$ with $i,j \in \lbrace 1, \dots, kN \rbrace$, as depicted in Figure~\ref{DimerGraph}. Then the probability distribution on tilings, or equivalently dimer configurations, of the corresponding weighted Aztec diamond is given through \eqref{ProbMeas}.    Here we have adopted the notation of \cite{BB23+}.
\vspace{-35pt}
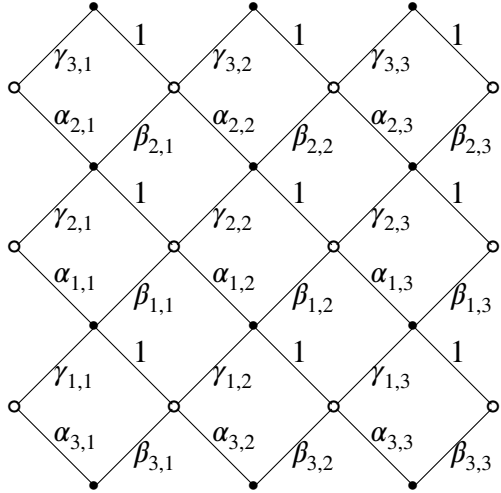
\begin{figure}[h]
\begin{tikzpicture}
    \put(120,-67){$\gamma_{3,1}$}
    \put(120,-127){$\gamma_{2,1}$}
    \put(120,-187){$\gamma_{1,1}$}

    \put(120,-90){$\alpha_{2,1}$}
    \put(120,-150){$\alpha_{1,1}$}
    \put(120,-210){$\alpha_{3,1}$}

    \put(150,-60){$1$}
    \put(150,-120){$1$}
    \put(150,-180){$1$}

    \put(150,-98){$\beta_{2,1}$}
    \put(150,-158){$\beta_{1,1}$}
    \put(150,-218){$\beta_{3,1}$}


    \put(180,-67){$\gamma_{3,2}$}
    \put(180,-127){$\gamma_{2,2}$}
    \put(180,-187){$\gamma_{1,2}$}

    \put(180,-90){$\alpha_{2,2}$}
    \put(180,-150){$\alpha_{1,2}$}
    \put(180,-210){$\alpha_{3,2}$}

    \put(210,-60){$1$}
    \put(210,-120){$1$}
    \put(210,-180){$1$}

    \put(210,-98){$\beta_{2,2}$}
    \put(210,-158){$\beta_{1,2}$}
    \put(210,-218){$\beta_{3,2}$}


    \put(240,-67){$\gamma_{3,3}$}
    \put(240,-127){$\gamma_{2,3}$}
    \put(240,-187){$\gamma_{1,3}$}

    \put(240,-90){$\alpha_{2,3}$}
    \put(240,-150){$\alpha_{1,3}$}
    \put(240,-210){$\alpha_{3,3}$}

    \put(270,-60){$1$}
    \put(270,-120){$1$}
    \put(270,-180){$1$}

    \put(270,-98){$\beta_{2,3}$}
    \put(270,-158){$\beta_{1,3}$}
    \put(270,-218){$\beta_{3,3}$}
\end{tikzpicture}
\vspace{-35pt}
\hspace*{-8pt}
   	\begin{center}
	\rotatebox{45}{
 \begin{tikzpicture}[scale=1.5]
\foreach \j in {1,...,3}{
    \foreach \k in {0,...,3}{
        \node
        (\j, \k) at (0.5-4.0+\j+\k,-0.5-\j+\k) {$\bullet$};
    }
}

\foreach \j in {1,...,4}{
    \foreach \k in {0,...,2}{
        \node
        (\j, \k) at (0.5-4.0+\j+\k,0.5-\j+\k) {$\circ$};
    }
}

\foreach \j in {1,...,3}{
    \foreach \k in {0,...,2}{
       \draw (1.5-4.0+\j+\k,-0.47-\j+\k)--(1.5-4.0+\j+\k,0.5-\j+\k);
        \draw (1.5-4.03+\j+\k,-0.5-\j+\k)--(1.5-5+\j+\k,-0.5-\j+\k);
    }
}

\foreach \j in {0,...,2}{
    \foreach \k in {0,...,2}{
       \draw (1.5-4.0+\j+\k,-0.53-\j+\k)--(1.5-4.0+\j+\k,-1.5-\j+\k);
        \draw (1.5-3.97+\j+\k,-0.5-\j+\k)--(1.5-3+\j+\k,-0.5-\j+\k);
    }
}

\end{tikzpicture}}
	\end{center}
	\vspace{-70pt}
\caption{Weighted Aztec diamond graph of size $3$ \label{DimerGraph}}
\end{figure}
To define the $k$-periodic Aztec diamond, we have to additionally assume that all weights are periodic in their first subscript with period $k$. That is $\alpha_{j+nk,i} = \alpha_{j,i}$, $n \in \mathbb Z$, and analogously for $\beta_{j,i}$ and $\gamma_{j,i}$. We refer to this model as the \emph{$k$-periodic Aztec diamond} and it is the subject of Section~\ref{SectkPer}.

While at first sight this model appears more general than the $k \times \ell$-periodic model (i.e., $\alpha_{j,i+n\ell} = \alpha_{j,i}$ etc.) introduced in \cite{BB23+}, this is not the case as long as we do not consider the limit $N \to \infty$. The reason is that for any fixed $N$, any $k$-periodic model of the Aztec diamond is a $k \times kN$-periodic model in the sense of \cite{BB23+}. Conversely, any $k \times \ell$-periodic model is trivially a $k$-periodic model. The distinction only becomes apparent if we consider the limit $N \to \infty$, in which case the genus of the spectral curve of a generic $k$-periodic model would increase with $N$. We are however not considering this limit in the present paper. 

In Section~\ref{2x2Section} we consider the generic $2 \times 2$-periodic Aztec diamond, meaning that the weights satisfy $\alpha_{j+2m, i+2n} = \alpha_{j, i}$, for $m,n \in \mathbb Z$, and analogously for $\beta_{j,i}$, $\gamma_{j,i}$. Here \emph{generic} means that we assume that the underlying spectral curve has genus one, which is equivalent to the presence of a smooth phase. Note that this model corresponds exactly to the $k \times \ell$-periodic model defined in \cite{BB23+}, with $k = \ell = 2$.

\subsection{Overview of the rest of the paper}

In the following we will briefly summarize the content of each section and lay out the structure of the paper. 
\begin{itemize}
    \item In Section~\ref{SoR} we include basic definitions and present our main results.   Sections~\ref{MVOPs_WH} and \ref{WH_MVOPs} deal with the relation between MVOPs with respect to a rational weight matrix and certain Wiener--Hopf factorizations of that weight matrix. While these sections are inspired by the appearance of MVOPs and Wiener--Hopf factorizations in tiling models, see \cite{BD19, DK21}, we do not assume any particular form of weight matrix other than it being rational.

    Section~\ref{SectkPer} discusses the $k$-periodic Aztec diamond and the associated matrix-orthogonality. Building on previous results from Proposition~\ref{prop3}, an alternative derivation of the double contour integral formula of Berggen and Duits \cite[Theorem 3.1]{BD19} in the case of the $k$-periodic Aztec diamond is given under milder assumptions on the weight structure. 

    In Section~\ref{2x2Section} we narrow down our choice of weight matrix further by considering the generic $2\times 2$-periodic model of the Aztec diamond. Here our main contribution is Theorem \ref{theo21}, which contains explicit formulas for the monic matrix-valued orthogonal polynomials (alternatively, the Wiener--Hopf factors) in terms of Jacobi theta functions. The underlying genus one Riemann surface is the spectral curve of the weight matrix.   
    \item Section~\ref{ProofsProp2} contains the proofs of the results stated in the Section~\ref{MVOPs_WH}--\ref{SectkPer}.
    \item Section~\ref{Sect2x2Const} contains the proof of Theorem \ref{theo21} on the explicit formulas for the matrix-valued orthogonal polynomials related to the $2\times2$-periodic model.  
    The necessary definitions are contained in Section~\ref{Sect_Prelimianries}, which also includes a few technical lemmas. In Section~\ref{Sect_Aux_Lem} more technical lemmas are proven.   Building on these preliminary results, the complete proof of Theorem \ref{theo21} is given in Section~\ref{SectProofThm}.
\end{itemize}
\section{Statement of results}\label{SoR}

In Sections~\ref{MVOPs_WH} and \ref{WH_MVOPs} we consider MVOPs with respect to a rational weight matrix $W$, and describe conditions under which these polynomials can be written in terms of certain Wiener--Hopf factorizations of $W$ (and vice versa). While the main motivation for this analysis stems from the papers \cite{BD19, DK21} on Wiener--Hopf factorizations resp.~MVOPs related to doubly periodic models of the Aztec diamond, we can and will assume a more general setting for now. 

In Sections~\ref{SectkPer} and \ref{2x2Section} we specialize to the Aztec diamond.
Our main result is Theorem \ref{theo21} which gives an explicit formula for the MVOP associated
with the $2 \times 2$ periodic Aztec diamond.

\subsection{MVOPs from Wiener--Hopf}\label{MVOPs_WH}

Matrix-valued orthogonal polynomials (MVOPs) arise in random tiling models
with periodic weights, as shown in \cite{DK21}. The underlying matrix-valued inner product of two $k\times k$ matrix-valued polynomials $P$ and $Q$ is given by
\begin{align}\label{innerProd}
    \langle \, P \, , \, Q \, \rangle := \frac{1}{2\pi i} \oint_\Gamma P(z) W(z) Q(z) dz \in \mathbb C^{k\times k},
\end{align}
where $W(z)$ is a rational matrix-valued function of size $k \times k$ with  $k \geq 1$
the vertical periodicity of the model. The integral \eqref{MVOP1} is taken entrywise along a simple closed contour $\Gamma$
in the complex plane that encloses the poles of $W$. By a pole of $W$ we mean a pole of at least one of the entries of $W$. In applications to random tilings, the weight matrix $W$ will depend explicitly on the parameters of the model, i.e., the periodic weights as in \eqref{defWforAD}. 

In the setting of doubly periodic models of the Aztec diamond, the $N$-th monic matrix-valued orthogonal polynomial
$P_N(z) = z^N I_k + O(z^{N-1})$, satisfying
\begin{equation} \label{MVOP1} 
	\frac{1}{2\pi i} \oint_{\Gamma} P_N(z) W(z) z^j dz = 0_k, \qquad j =0,1, \ldots, N-1 
	\end{equation}
 is of special interest. Here, $N$ is related to the size $kN$ of the $k$-periodic Aztec diamond. 
 
 Motivated by their appearance in tilings models, we will study in the present and following subsections the orthogonality condition \eqref{MVOP1} under the  minimal assumptions on $W$, $\Gamma$, $N$ stated above, i.e., no relation to tiling models is assumed for the moment. In other words, $W$ is \emph{any} rational $k \times k$ matrix-valued function, $\Gamma$ is only required to enclose all poles of $W$ and $N \in \mathbb N$ is arbitrary. Furthermore, we will be mostly concerned with the case $k \geq 2$ which deals with genuinely matrix-valued, rather than scalar-valued, orthogonality. Just as their scalar-valued counterpart, MVOPs are characterized by a Riemann-Hilbert problem, see \cite{DK21}.

Berggren and Duits obtained in \cite{BD19} a Wiener--Hopf factorization of certain weight matrices $W$, in particular those related to the periodic Aztec diamond.
Here a Wiener--Hopf factorization on a simple closed contour $\Gamma$ (not necessarily
the same one as in \eqref{MVOP1})  is a factorization 
\begin{equation} \label{Wfactors} 
	W  = \varphi_- \varphi_+ = \widehat{\varphi}_+ \widehat{\varphi}_-  \qquad \text{ on } \Gamma, \end{equation}
where $\varphi$ and $\widehat{\varphi}$  are invertible matrix-valued functions that are analytic
on $\mathbb C \setminus \Gamma$ with boundary values $\varphi_{\pm}$ and $\widehat{\varphi}_{\pm}$
on $\Gamma$ that are also invertible, and that are normalized such that
\begin{equation} \label{phiasymp} 
	\lim_{z \to \infty} z^N \varphi(z) = I_k, \quad
	\lim_{z \to \infty} z^N \widehat{\varphi}(z) = I_k.
\end{equation}
Throughout this paper $\Gamma$ is a simple closed contour with
 positive orientation that encloses a bounded domain $\Omega_{int}$.
 We use $\Omega_{ext}$ to denote the unbounded component of $\mathbb C \setminus \Gamma$, cf.~Figure~\ref{Fig:Gamma}.
 
Here we slightly deviate from \cite{BD19} in two ways.
\begin{itemize}
	\item The factorization \eqref{Wfactors} is stated in \cite{BD19} for the unit circle
	$\Gamma_0 = \{ z \in \mathbb C \mid |z| = 1 \}$ only, and
	\item  $\varphi$ is called $\widetilde{\phi}_{\pm}$ and 
	$\widehat{\varphi}$  is called $\phi_{\pm}$ in \cite{BD19},
	where the functions $\phi_{\pm}$ and $\widetilde{\phi}_{\pm}$ are used
	to denote analytic matrix-valued functions, and not just their  boundary
	values on $\Gamma$. 
	 We reserve  $\varphi_{\pm}$ and $\widehat{\varphi}_{\pm}$ 
	for the boundary values, since we will be dealing with Riemann-Hilbert problems where it is customary to do so.
\end{itemize}

If a Wiener--Hopf factorization on $\Gamma$ exists, then $W$ is invertible on $\Gamma$ since
we are assuming that the boundary values  $\varphi_{\pm}$ and $\widehat{\varphi}_{\pm}$ 
are invertible matrices. It follows that $\det W$ is a rational function that is not
identically zero, hence $W^{-1}$ is a rational function as well. 

In a typical situation
we will assume that $W$ and $W^{-1}$ have no poles in common, and $\Gamma$ separates
the poles of $W$ from the poles of $W^{-1}$ in the sense that the poles of $W$ are in the interior
region $\Omega_{int}$ and the poles of $W^{-1}$ are in the exterior region $\Omega_{ext}$. As not all results of the present paper require this assumption, we will only state it when needed. In the case of the doubly periodic  Aztec diamond, the poles of $W$ and $W^{-1}$ will be restricted to the real line, see Figure~\ref{Fig:Gamma}. 
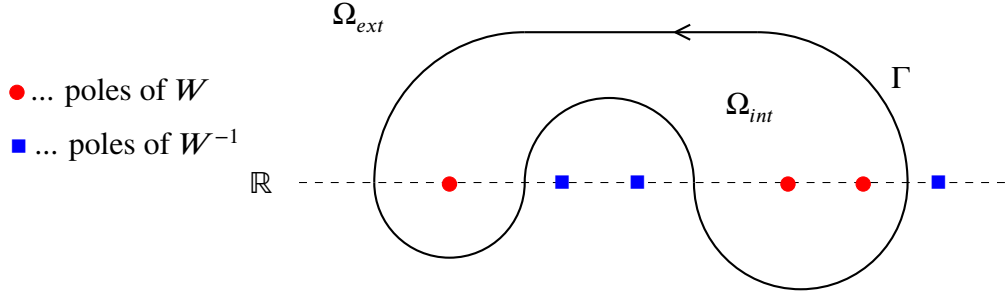
\begin{figure}
            \centering
            \begin{tikzpicture}
                \draw[dashed] (2,0)--(11.4,0);

                \draw[red,fill=red] (4,-0.02) circle (.5ex);
                \draw (5.5,0) node {\textcolor{blue}{$\smblksquare$}};
                \draw (6.5,0) node {\textcolor{blue}{$\smblksquare$}};
                 \draw[red,fill=red] (8.5,-0.02) circle (.5ex); 
                 \draw[red,fill=red] (9.5,-0.02) circle (.5ex);
                \draw (10.5,0) node {\textcolor{blue}{$\smblksquare$}};

                \draw (1.5,0) node {$\mathbb R$};

                \draw[thick] (3,0) arc (-180:0:1);
                \draw[thick] (5,0) arc (180:0:1.125);
                \draw[thick] (7.25,0) arc (-180:0:1.42);
                \draw[thick] (3,0) arc (180:90:2);
                \draw[thick] (10.09,0) arc (0:90:2);
                \draw[thick] (5,2)--(8.09,2);
                
                \draw[thick] (7,2)--(7.2,2.1);
                \draw[thick] (7,2)--(7.2,1.9);

                \draw (8,1) node {$\Omega_{int}$};
                \draw (2.8,2.2) node {$\Omega_{ext}$};

   		\draw[red,fill=red] (-1.75,1.2) circle (.5ex);  
   		\draw (-0.5,1.2) node 	{\textcolor{red}{$\bullet$} ... poles of $W$};
                \draw (-0.3,0.5) node {\textcolor{blue}{$\smblksquare$} ... poles of $W^{-1}$};

                 \draw (10,1.4) node {$\Gamma$};
            \end{tikzpicture}
            \caption{The closed contour $\Gamma$ and the domains $\Omega_{int}$ and $\Omega_{ext}$ for the $k$-periodic Aztec diamond considered in Section~\ref{SectkPer}.}
            \label{Fig:Gamma}
        \end{figure}

Matrix-valued orthogonality and Wiener--Hopf factorizations turn out to be intimately related in case that the MVOP of degree $N$ has the peculiar property that $P_N W$ is a matrix-valued polynomial as well,
which means that the multiplication by $P_N$ from the left cancels the poles of $W$. 
In that case the matrix-valued orthogonality \eqref{MVOP1} takes the stronger form
 \begin{equation} \label{MVOP2} 
 	\frac{1}{2\pi i} \oint_{\Gamma} P_N(z) W(z) z^j dz = 0_k, \qquad j =0,1,2, \ldots,
 \end{equation}
 simply by Cauchy's theorem, as all entries of $P_NW$ are polynomials. Thus
 the integral \eqref{MVOP1} vanishes not only for $j$ up to $N-1$,
 but for all non-negative integers $j$. This property turns out to hold for the
 weight matrices \eqref{defWforAD} coming from the $k$-periodic Aztec diamond. 
 However, it does not hold
 for the weight matrices in periodic hexagon tiling models, see e.g.~\cite{GK21}. 
 
 Together with the matrix orthogonality \eqref{MVOP1} we also consider the 
 matrix-valued polynomials with orthogonality on the right. That is, $\widehat{P}_N$ 
 is monic of degree $N$ with
 \begin{equation} \label{MVOP3} 
 	\frac{1}{2 \pi i} \oint_{\Gamma} W(z) \widehat{P}_N(z) z^j dz = 0, \qquad j=0,1, \ldots, N-1. 
 	\end{equation}
Note that unlike the convention of \cite{DK21}, we have not included a transpose for polynomials on the right of $W$, as this simplifies formulas later on. If $W \widehat{P}_N$ is a matrix polynomial, then \eqref{MVOP3} takes the stronger form
\begin{equation} \label{MVOP4} 
	\frac{1}{2 \pi i} \oint_{\Gamma} W(z) \widehat{P}_N(z) z^j dz = 0, \qquad j=0,1,2, \ldots. 
\end{equation}
 
Our first result states that if the Wiener--Hopf factorization on $\Gamma$ exists, then the MVOPs $P_N$ and $\widehat{P}_N$ exist and can be expressed in terms of the Wiener--Hopf factors. In addition $P_N W$ and $W \widehat{P}_N$ are matrix polynomials.

\begin{proposition} \label{prop1}
	Suppose the rational function $W$ has a Wiener--Hopf factorization \eqref{Wfactors}
	on a contour $\Gamma$ that encloses all the poles of $W$, but not any poles of $W^{-1}$.
	Then the following hold.
	\begin{enumerate}
		\item[\rm (a)] 
	$P_N$ defined by 
	\begin{equation} \label{defP}  
		P_N(z) = \begin{cases} \varphi(z)^{-1}, & \quad z \in \Omega_{ext}, \\
			\varphi(z) W^{-1}(z), & \quad z \in \Omega_{int},
	\end{cases} \end{equation}
	is a monic matrix-valued polynomial of degree $N$ with the property
	that $P_N W$ is a matrix-valued polynomial, and the matrix-valued orthogonality
	in the stronger form \eqref{MVOP2} holds,
	\item[\rm (b)]  $\widehat{P}_N$ defined by
	\begin{equation} \label{defhatP}  
		\widehat{P}_N(z) = \begin{cases} \widehat{\varphi}(z)^{-1}, & \quad z \in \Omega_{ext}, \\
			W^{-1}(z) \widehat{\varphi}(z), & \quad z \in \Omega_{int}, 
	\end{cases} \end{equation}
	is a monic matrix-valued polynomial of degree $N$ with the property that 
	$W \widehat{P}_N$ is a matrix-valued polynomial as well, and
	the matrix-valued orthogonality from the right in the stronger form \eqref{MVOP4} holds,
	\item[\rm (c)] $\det P_N = \det \widehat{P}_{N}$ and this is a polynomial
	of degree $kN$, whose zeros coincide with the $kN$ poles of $\det W$, and
	\item[\rm (d)] the above defined  $P_N$ and $\widehat{P}_N$ are the unique monic 
	matrix-valued polynomials of degree $N$
	satisfying the matrix-valued orthogonality \eqref{MVOP1} and \eqref{MVOP3}.
	\end{enumerate}
\end{proposition}
The proof of Proposition~\ref{prop1} is in Sections~\ref{proofprop1abc} and \ref{proofprop1d} below.
The proofs of parts (a) and (b) are an almost immediate  verification.
In particular, the matrix-valued orthogonality in the stronger forms \eqref{MVOP2} and \eqref{MVOP4} follows
as we already discussed. The proof of part (c) requires a bit work. 
The uniqueness statement in part (d) is less obvious, and we prove
it by means of the Riemann-Hilbert (RH) problem for MVOPs, \cite{CM12,GIM11}. It is an immediate 
extension of the RH problem for orthogonal polynomials due to Fokas, Its, Kitaev \cite{FIK92}. 

The aforementioned RH problem asks for an analytic matrix-valued function $Y : \mathbb C \setminus \Gamma \to \mathbb C^{2k \times 2k}$
satisfying 
\begin{equation} \label{RHYjump} 
	Y_+ = Y_- \begin{pmatrix} I_k & W  \\ 0_k & I_k \end{pmatrix} \quad \text{ on } \Gamma, 
\end{equation}
and the asymptotic condition
\begin{equation} \label{RHYasymp} 
	Y(z) = \left( I_{2k} + O(z^{-1}) \right) \begin{pmatrix} z^{N} I_k & 0_k \\ 0_k & z^{-N} I_k \end{pmatrix} 
	\quad \text{ as } z \to \infty. \end{equation}
As usual for RH problems the notation $Y_{+}(z)$ for $z$ on an oriented contour
denotes the limit of $Y(w)$ as $w \to z$ from the left, and $Y_-(z)$ denotes the limit
from the right. We assume $\Gamma$ is oriented counterclockwise.
It is part of the assumptions in the RH problem that these limits exist. 
The matrices in \eqref{RHYjump} and \eqref{RHYasymp} are block matrices with blocks of size
$k \times k$. We use $I_k$ to denote the identity matrix of size $k \times k$, and $0_k$
to denote the zero matrix of size $k \times k$.

Provided it exists, the unique solution of the RH problem is
\begin{equation} \label{RHYsol}
	Y(z) = \begin{pmatrix} P_N(z) & \ds \frac{1}{2\pi i} \oint_{\Gamma} \frac{P_N(s) W(s)}{s-z} ds \\[10pt]
		Q_{N-1}(z) & \ds \frac{1}{2\pi i} \oint_{\Gamma} \frac{Q_{N-1}(s) W(s)}{s-z} ds \end{pmatrix}, \quad
	z \in \mathbb C \setminus \Gamma,
\end{equation} 
where $Q_{N-1}$ is a matrix-valued polynomial of degree $\leq N-1$ satisfying
\begin{equation} \label{QNortho} 
	\frac{1}{2\pi i} \oint_{\Gamma} Q_{N-1}(z) W(z) z^j dz =
	\begin{cases} 0_k & \text{ for } j =0,1, \ldots, N-2, \\
		-I_k & \text{ for } j=N-1. \end{cases} \end{equation}
This is an immediate consequence of the Sokhotski--Plemelj formula. We will show  that in the situation
of Proposition \ref{prop1}, the RH problem has indeed a solution and the
 polynomial $Q_{N-1}$ is given by
\begin{equation} \label{QNformula2} 
	Q_{N-1}(z) = \begin{cases}
		\left( \ds \frac{1}{2\pi i} \oint_{\Gamma}
		\widehat{\varphi}_+^{-1}(s) \varphi_-(s)  \frac{ds}{s-z} \right)  \varphi^{-1}(z),
		& z \in \Omega_{ext}, \\
		\left( \ds \frac{1}{2\pi i} \oint_{\Gamma}
		\widehat{\varphi}_+^{-1}(s) \varphi_-(s)  \frac{ds}{s-z} \right)  \varphi(z) W^{-1}(z)
		- \widehat{\varphi}^{-1}(z),
		& z \in \Omega_{int},
	\end{cases}
\end{equation} 
see the proof of Proposition \ref{prop1}(d). In particular, for the choice of weight matrix \eqref{defWforAD} related to the periodic Aztec diamond, the RH problem has a unique solution which is necessarily of the form \eqref{RHYsol}, see \cite[Lemma 4.8]{DK21}.

\subsection{Wiener--Hopf from MVOPs}
\label{WH_MVOPs}

Now we want to start from the MVOP $P_N$ in \eqref{MVOP1} and we ask under which conditions the  Wiener--Hopf factorizations \eqref{Wfactors} can be obtained from it. From Proposition \ref{prop1} it is clear that a necessary condition is
that $\widehat{P}_N$ exists as well, and that both $P_NW$ and $W \widehat{P}_N$
are matrix-valued polynomials.
	
The existence of $\widehat{P}_N$ follows from the unique existence of $P_N$
and this is a general fact, since the conditions \eqref{MVOP1} and \eqref{MVOP3}
can be phrased as linear systems of equations with  the same block matrix. 
We will, however, prove this fact in Proposition \ref{prop2}(a) from the RH problem for
matrix-valued orthogonal polynomials.

A further object of interest is the reproducing kernel $R_N$ with respect to the inner product  \eqref{innerProd} for matrix-valued polynomials of degree $\leq N-1$, which satisfies 
\begin{itemize}
        \item $\langle \, F, \, R_{N}( \, \cdot \, , z) \, \rangle = F(z)$, for all matrix-valued polynomials $F$ of 
        
        degree $\leq N-1$, 
        \item $\langle \,  R_{N}(w,  \, \cdot \, ) , \, F \, \rangle = F(w)$, for all matrix-valued polynomials $F$ of 
        
        degree $\leq N-1$,
        \item $R_N(w,z)$ is a polynomial of degree $\leq N-1$ in both variables $w$, $z$.
    \end{itemize}
As shown by Delvaux in \cite{D10}, provided the RH problem has a solution, the reproducing kernel can be expressed via a Christoffel--Darboux type formula  
\begin{equation} \label{RHYRN}
	R_N(w,z) = \frac{1}{z-w} \begin{pmatrix} 0_k & I_k \end{pmatrix}
		Y^{-1}(w) Y(z) \begin{pmatrix} I_k \\ 0_k \end{pmatrix}. 	
	\end{equation}
The significance of the reproducing kernel comes from its appearance in the contour integral formula \eqref{phiRphi}. We provide an integral representation of $R_N$ in Proposition~\ref{prop3}(d) below.

We first state  a partial converse to Proposition \ref{prop1}. 
\begin{proposition} \label{prop2}
	Let $W$ be a rational matrix-valued function and let $\Gamma$
	be a closed contour that encloses the poles of $W$. 
	Suppose the monic matrix-valued polynomial $P_N$ of degree $N$ with the 
	orthogonality \eqref{MVOP1} uniquely exists. \begin{enumerate}
		\item[\rm (a)] 
	Then also  $\widehat{P}_N$ uniquely exists and it is given by
	\begin{equation} \label{RHYhatP} 
		\widehat{P}_N(z) = \begin{pmatrix} 0_k & I_k \end{pmatrix}
		Y^{-1}(z) \begin{pmatrix} 0_k \\ I_k \end{pmatrix},
		\quad z \in \mathbb C \setminus \Gamma, \end{equation}
	where $Y$ is the solution of the RH problem \eqref{RHYjump}-\eqref{RHYasymp}.
	\end{enumerate}
	Assume, in addition, that both $P_NW$ and $W \widehat{P}_N$ are matrix-valued polynomials as well.
	\begin{enumerate}
		\item[\rm (b)] Then  $P_N(z)$ and  $\widehat{P}_N(z)$ are invertible for $z \in \Omega_{ext}$, and
		\item[\rm (c)]  
		$\widehat{\varphi}$ and  $\varphi$  defined by 
		\begin{equation} \label{defphi}
		\widehat{\varphi}(z) = \begin{cases} \widehat{P}_N(z)^{-1}, & \quad z \in \Omega_{ext}, \\
			W(z) \widehat{P}_N(z), & \quad z \in \Omega_{int},
			\end{cases} \end{equation}
		and
		\begin{equation} \label{deftilphi}
			\varphi(z) = \begin{cases} P_N(z)^{-1}, & \quad z \in \Omega_{ext}, \\
				P_N(z) W(z), & \quad z \in \Omega_{int}, 
		\end{cases} \end{equation}
		are analytic  in $\mathbb C \setminus \Gamma = \Omega_{ext} \cup \Omega_{int}$
		and satisfy \eqref{Wfactors} and \eqref{phiasymp}.  
	\end{enumerate}
\end{proposition}
Part (a) is not new as it is essentially contained in \cite{BFM21}.
Proposition \ref{prop2}(c) does not give a proper Wiener--Hopf factorization yet,
since for that  $\varphi$ and $\widehat{\varphi}$  would have to be invertible in $\mathbb C \setminus
\Gamma$. The invertibility in $\Omega_{ext}$ is immediate from \eqref{defphi} and \eqref{deftilphi}.
The invertibility in $\Omega_{int}$ depends on the choice of the contour $\Gamma$,
as $\Gamma$ should not  enclose any of the zeros of $P_N W$ and $W \widehat{P}_N$. 
Under that extra assumption the following holds.

\begin{proposition} \label{prop3}
	If, in addition to the assumptions in Proposition \ref{prop2}, 
	the scalar polynomial $\det\left(P_NW\right)$ has
	no zeros in $\Gamma \cup \Omega_{int}$, then also the following hold.
	\begin{enumerate}
	\item[\rm (a)] $\det\left(W \widehat{P}_N\right)$ has no zeros in $\Gamma \cup \Omega_{int}$,
	
	\item[\rm (b)]  
		$\varphi$ and $\widehat{\varphi}$  given by \eqref{defphi} and \eqref{deftilphi} 
		are invertible in $\mathbb C \setminus \Gamma$, and have
		boundary values $\varphi_{\pm}$ and $\widehat{\varphi}_{\pm}$ that
		are also invertible, hence the Wiener--Hopf factorization of $W$ on $\Gamma$ exists,
	\item[\rm (c)] the matrix-valued polynomial $Q_{N-1}$ from \eqref{RHYsol}
		is given by 
		\begin{equation} \label{QNformula3} 
			Q_{N-1}(z) = \begin{cases}
				\left( \ds \frac{1}{2\pi i} \oint_{\Gamma}
				\widehat{P}_N^{-1}(s) W^{-1}(s) P_N^{-1}(s)  \frac{ds}{s-z} \right) P_N(z),
			    \qquad \hfill{z \in \Omega_{ext}}, \\
				\left( \ds \frac{1}{2\pi i} \oint_{\Gamma}
				\widehat{P}_N^{-1}(s) W^{-1}(s) P_N^{-1}(s)  \frac{ds}{s-z} \right) P_N(z) \\
				\hfill{- \widehat{P}_N^{-1}(z) W^{-1}(z), \quad z \in \Omega_{int}},
			\end{cases}
		\end{equation} 
		and
	\item[\rm (d)] the reproducing kernel \eqref{RHYRN} is given  by
	\begin{multline} \label{RNformula}
	 R_N(w,z)  \\ = \widehat{P}_N(w) 
	 \left[ \frac{1}{2\pi i} \oint_{\Gamma}
	\widehat{P}_N^{-1}(s) W^{-1}(s) P_N^{-1}(s) \frac{ds}{(s-z)(s-w)} \right] P_N(z), \\
	 w,z \in \Omega_{ext}.
	\end{multline}
	\end{enumerate}
\end{proposition}
Using the formulas \eqref{defphi} and \eqref{deftilphi} in \eqref{QNformula3} and \eqref{RNformula}
one could alternatively express $Q_{N-1}$ and $R_N$ in terms of the Wiener--Hopf 
factors $\varphi$ and $\widehat{\varphi}$.
The formula \eqref{RNformula} gives $R_N(w,z)$ for $w, z \in \Omega_{ext}$. 
Its  analytic continuation to $\Omega_{int}$ can be easily obtained from it using the
Sokhotskii--Plemelj formula, but is not needed for the purposes of the present paper.
Propositions~\ref{prop2} and \ref{prop3} are proved in Sections~\ref{proofprop2}
and \ref{proofprop3}.

\subsection{The $k$-periodic Aztec diamond}	
\label{SectkPer}
We will focus now on a more particular choice of weight matrix $W$. Namely, in the setting of the Berggren-Borodin paper \cite{BB23+}, that deals with the doubly periodic
Aztec diamond with periods $k$ and $\ell$, the weight matrix $W$
is a product of matrices of the form
\begin{equation} \label{phib} \phi^b(z; \vec{\alpha}, \vec{\gamma})
	= \begin{pmatrix} \gamma_1 & 0 & \cdots & 0 & \alpha_k z^{-1} \\
		\alpha_1 & \gamma_2 & \cdots & 0 & 0  \\
		\vdots & \vdots & \ddots & \vdots & \vdots \\
		0 & 0 & \cdots & \gamma_{k-1} & 0 \\
		0 & 0 & \cdots & \alpha_{k-1} & \gamma_k \end{pmatrix} \end{equation}
and
\begin{equation} \label{phig} \phi^g(z; \vec{\beta}) = \frac{1}{1- \prod_{j=1}^k \beta_j z^{-1}} \begin{pmatrix}
		1 & \prod_{j=2}^k \beta_j z^{-1} & \cdots & \beta_k z^{-1} \\
		\beta_1 & 1 & \cdots & \beta_k \beta_1 z^{-1} \\
		\vdots & \vdots & \ddots & \vdots \\
		\prod_{j=1}^{k-1} \beta_j & \prod_{j=2}^{k-1} \beta_j & \cdots & 1 \end{pmatrix}.
\end{equation}
The matrix \eqref{phib} depends on two vectors $\vec{\alpha} = (\alpha_1, \ldots, \alpha_k)$
and $\vec{\gamma} = (\gamma_1, \ldots, \gamma_k)$ and the matrix \eqref{phig} depends
on one vector $\vec{\beta} = (\beta_1, \ldots, \beta_k)$. All  parameters
are assumed to be positive real numbers. 

We choose $N \in \mathbb N$, together with vectors  
$\vec{\alpha}_1, \ldots \vec{\alpha}_{kN}$,  
$\vec{\gamma}_1, \ldots \vec{\gamma}_{kN}$, 
$\vec{\beta}_1, \ldots \vec{\beta}_{kN}$ which define the weights for the corresponding $k$-periodic Aztec diamond model, see Figure~\ref{DimerGraph}. We proceed to form the matrix product
\begin{equation} \label{defWforAD} 
	W(z) = \phi^b(z;\vec{\alpha}_1, \vec{\gamma}_1) \phi^g(z; \vec{\beta}_1)
	\cdots	
	\phi^b(z;\vec{\alpha}_{kN}, \vec{\gamma}_{kN}) \phi^g(z; \vec{\beta}_{kN}).
\end{equation}
Then $W$ is a rational matrix-valued function. For generic choices of parameters, the genus of the spectral curve of $W$ will grow with $N \to \infty$, which poses a major challenge in the asymptotic analysis of the underlying tiling model. A significant simplification occurs in case the three sequences $(\vec{\alpha}_{i})_{i}$, $(\vec{\beta}_{i})_i$, 
$(\vec{\gamma}_{i})_i$ with $i = 1, \dots, kN$ are periodic
with period $\ell$. This corresponds to the aforementioned $k \times \ell$-periodic Aztec diamond considered in \cite{BB23+}, where a thorough asymptotic analysis was performed.

In the present subsection we do not assume that $(\vec{\alpha}_{i})_{i}$, $(\vec{\beta}_{i})_i$, 
$(\vec{\gamma}_{i})_i$,  are periodic, that is, we will consider an Aztec diamond
with weights that are only periodic with period $k$ in one direction. We refer to this model as the $k$-periodic Aztec diamond. For fixed $N$ it is a special case of the $k \times \ell$-periodic model, with $\ell = kN$.  Using \eqref{RNformula} we will now derive a double integral formula for the underlying correlation kernel, relaxing the earlier restriction \eqref{polezero} on the weights assumed in \cite{BB23+, BD19}. However, we presently do not know how to perform an adequate steepest descent analysis for these integrals as $N \to \infty$ in the case of the $k$-periodic Aztec diamond.

Following \cite{BB23+}, let us denote the vector entries of $\vec{\alpha}_i$ by $\alpha_{j,i} =(\vec{\alpha}_i)_j$ with $j = 1, \dots, k$, $i = 1, \dots, kN$ and analogously for $\vec{\beta}_i$ and $\vec{\gamma}_i$. By \cite[Lemma 2.3]{BB23+} one has $\det \phi^b(z; \vec{\alpha}_i, \vec{\gamma}_i)
= \gamma^v_i - (-1)^k \alpha^v_i z^{-1}$ and $\det \phi^g(z;\vec{\beta}_i) = (1-\beta^v_i z^{-1})^{-1}$
where $\alpha^v_i = \prod_{j=1}^k \alpha_{j,i}$, $\beta^v_i = \prod_{j=1}^k \beta_{j,i}$
and $\gamma^v_i = \prod_{j=1}^k \gamma_{j,i}$.
Hence by \eqref{defWforAD}
\begin{equation} \label{detWforAD} 
	\det W(z) = \prod_{i=1}^{kN} \frac{\gamma_i^v z - (-1)^k \alpha^v_i}{z-\beta^v_i}. \end{equation}
Berggren and Duits determined in \cite{BD19} conditions under which the Wiener--Hopf factorization of \eqref{defWforAD} with $\gamma_{j,i} \equiv 1$
exists on the unit circle $\Gamma_0 : |z|=1$. For general $\gamma_{j,i} >0$ this condition reads  
\begin{equation} \label{polezero} 
	\beta_i^v < 1 < \frac{\alpha_i^v}{\gamma_i^v} \quad \text{for } i =1, \ldots, kN,
\end{equation}
see \cite[Assumption 4.1(b)]{BB23+}. The proof of this result is purely algebraic and is based on certain switching rules for the matrices of type $\phi^b$ and $\phi^g$, which were subsequently related to the domino shuffle by Chhita and Duits in \cite{CD23}. For the matrices of type \eqref{phib}, \eqref{phig} the switching rule reads
\begin{align*}
    \phi^b(z; \vec{\alpha}, \vec{\gamma}) \phi^g(z; \vec{\beta}) =  \phi^g(z; \vec{\beta'})\phi^b(z; \vec{\alpha'}, \vec{\gamma})
\end{align*}
where
\begin{align*}
    \alpha_j' = \alpha_{j-1} \frac{\alpha_j+\gamma_{j+1}\beta_j}{\alpha_{j-1}+\gamma_j \beta_{j-1}}, \qquad \beta_j' = \beta_{j-1} \frac{\alpha_j + \gamma_{j+1}\beta_j}{\alpha_{j-1}+\gamma_j \beta_{j-1}}.
\end{align*}
Notice that $\vec{\gamma}$ remains unchanged under the switching rule. In case \eqref{polezero} is not satisfied the switching rules still hold (as it is an algebraic identity), but the poles and zeros of \eqref{defWforAD} are no longer separated by the unit circle. Consequently, the Wiener--Hopf factorization of $W$ exists only on appropriate contours $\Gamma$ separating these points, cf.~Figure~\ref{Fig:Gamma}.
\begin{proposition} \label{prop4} 
	Suppose the poles $\beta_i^v$, $i=1,\ldots, kN$, of $\det W$ 
	are distinct from its zeros $(-1)^k \alpha_i^v/\gamma_i^v$, $i=1, \ldots, kN$.
	Let $\Gamma$ be a simple closed contour that encloses
	the poles with the zeros in the exterior. Then the
	Wiener--Hopf factorization \eqref{Wfactors}- \eqref{phiasymp}
	of $W$ on $\Gamma$ exists.
\end{proposition} 

The proof is a straightforward adaptation of the proof found in \cite{BD19} (see \cite{BB23+} for the case of general $\gamma_{i,j} > 0$) and shall be omitted. The motivation for including the condition \eqref{polezero} in the papers \cite{B21, BB23+, BD19} was related to the double integral formula \eqref{DIntPhi} for the correlation kernel. The proof of formula \eqref{DIntPhi} in \cite[Theorem 3.1]{BD19}, which is based on RH techniques, made explicit use of the assumption that the Wiener--Hopf factorization is on the unit circle. Condition  \eqref{polezero} also plays a role in the asymptotic analysis in \cite{BB23+} (see Assumption 4.1(b) therein for a discussion of its significance). No such condition is necessary for the purposes of our paper. In fact, we will now give a direct proof of the double integral formula which is based on the formula \eqref{RNformula} for the reproducing kernel and does not require  assumption \eqref{polezero}.

As mentioned earlier, the reproducing kernel \eqref{RNformula} appears in a double contour integral 
that is part of a formula for the correlation kernel of the determinantal point process. Concretely, as shown in \cite[Theorem 4.7]{DK21} the correlation kernel $K(\, \cdot \, , \, \cdot \, )$, which describes domino correlations of random tilings of the $k$-periodic Aztec diamond, is given by
\begin{align}\label{phiRphi}
\big[&K(x, ky +j; x', ky'+j') \big]_{j,j' = 0}^{k-1}= -\frac{\chi_{x>x'}}{2\pi i} \oint_{\Gamma'} \left(\prod_{i=x'+1}^{x} \phi_i(z)\right)z^{y'-y} \frac{dz}{z}
\\\nonumber
&+\frac{1}{(2\pi i)^2}
\oint_{\Gamma'} \oint_{\Gamma'}
\left(\prod_{i=x'+1}^{2 k  N} \phi_i(z_1)\right) R_N(z_1,z_2)
\left(\prod_{i=1}^{x} \phi_i(z_2) \right) z_1^{y'} z_2^{-y} \frac{dz_1 dz_2}{z_2},
\\ \nonumber 
&\hspace{5.5cm} 0 \leq y, y'\leq N-1, \ 0 < x, x'<2kN-1,
\end{align}
where $\phi_{2i-1} = \phi^b(\, \cdot \, ; \vec{\alpha}_i, \vec{\gamma}_i)$ 
and $\phi_{2i} = \phi^g(\, \cdot \, ; \vec{\beta}_i)$ for $i=1,2, \ldots, kN$. 
The  closed contour $\Gamma'$ goes around $0$ and all the poles of $W$, i.e., the real numbers $\beta_i^v$. We may and do assume
that $\Gamma'$ lies in the exterior domain $\Omega_{ext}$, in particular $\Gamma$ lies fully in the interior of $\Gamma'$. The relation between the determinantal point process induced by \eqref{phiRphi} and random tilings of the $k$-periodic Aztec diamond is based on a bijection between tilings and non-intersecting paths, and can be found in \cite[Section~3]{DK21}. As analysing the double contour integral in \eqref{phiRphi} poses the main challenge, we will try to simplify it in the following, culminating in Proposition \ref{PropDoubleInt} below.

Replacing $R_N(z_1,z_2)$ in \eqref{phiRphi} by the contour integral \eqref{RNformula} results in a threefold
integral. We interchange the order of integration and focus on the $z_1$ integral first. This amounts
to evaluating
\begin{equation} \label{RNz1integral} \frac{1}{2\pi i} \oint_{\Gamma'}
	\left(\prod_{i=x'+1}^{2 k  N} \phi_i(z_1)\right)
	\widehat{P}_N(z_1)  z_1^{y'} \frac{dz_1}{s-z_1}, \qquad s \in \Gamma. 
\end{equation}
The function
\[ z \mapsto \left(\prod_{i=x'+1}^{2 k N} \phi_i(z)\right)
\widehat{P}_{N}(z) z^{y'}
\]
has possible poles at $\beta_{i}^v$ for $i= \lceil \frac{x'+1}{2} \rceil, \ldots, kN$, 
since $\beta_i^v$ is the pole of $\phi_{2i}$. There is no pole at $z=0$ since
we assume $y' \geq 0$. The singularities $\beta_i^v$ are however removable, because by \eqref{defphi}
and \eqref{defWforAD}
\[ \widehat{P}_{N}(z) = W^{-1}(z) \widehat{\varphi}(z)
= \left(\prod_{i=1}^{2kN} \phi_i(z) \right)^{-1} \widehat{\varphi}(z), \quad z \in \Omega_{int}, \]
and therefore
\[  \left(\prod_{i=x'+1}^{2 k N} \phi_i(z)\right)
\widehat{P}_{ N}(z) =  \left(\prod_{i=1}^{x'} \phi_i(z)\right)^{-1} \widehat{\varphi}(z) \]
which has no poles in $\Omega_{int}$.

Thus in \eqref{RNz1integral} there is only a residue 
contribution from the pole at $z_1=s$ and we obtain 
\[ \frac{1}{2\pi i} \oint_{\Gamma'}
\left(\prod_{i=x'+1}^{2 kN} \phi_i(z_1)\right)
\widehat{\varphi}^{-1}(z_1)  z_1^{y'} \frac{dz_1}{s-z_1}
= 
- \left(\prod_{i=1}^{x'} \phi_i(s)\right)^{-1} \widehat{\varphi}_+(s) s^{y'}, 	
\qquad s \in \Gamma. \]

The remaining double integral is (with $z$ instead of $z_2$)
\begin{align}\label{5xphi}  - \frac{1}{(2\pi i)^2}
\oint_{\Gamma'} \oint_{\Gamma} \left(\prod_{i=1}^{x'} \phi_i(s)\right)^{-1} \widehat{\varphi}_+(s) \widehat{\varphi}_-(s)
\varphi_+^{-1}(s) \varphi^{-1}(z)
\left(\prod_{i=1}^{x} \phi_i(z) \right) \frac{s^{y'}}{z^{y}} \frac{ds dz}{z(s-z)}.  
\end{align}
From \eqref{Wfactors} we have
\[ \widehat{\varphi}_+(s) \widehat{\varphi}_-(s) \varphi_+^{-1}(s) = \varphi_-(s), 
\qquad s \in \Gamma, 
\]
which can be substituted into \eqref{5xphi}, giving us the following proposition.
\begin{proposition} \label{PropDoubleInt}
    The correlation kernel \eqref{phiRphi} can be written as
\begin{align}  \label{DIntPhi}
\big[&K(x, ky +j; x', ky'+j') \big]_{j,j' = 0}^{k-1}= -\frac{\chi_{x>x'}}{2\pi i} \oint_{\Gamma'} \left(\prod_{i=x'+1}^{x} \phi_i(z)\right)z^{y'-y} \frac{dz}{z}
\\\nonumber
&- \frac{1}{(2\pi i)^2}
\oint_{\Gamma'} \oint_{\Gamma} \left(\prod_{i=1}^{x'} \phi_i(s) \right)^{-1}  \varphi_-(s) \varphi^{-1}(z)
\left(\prod_{i=1}^{x} \phi_i(z) \right) \frac{s^{y'}}{z^{y}} 
\frac{ds dz}{z(s-z)},
\\ \nonumber 
&\hspace{5.5cm} 0 \leq y, y'\leq N-1, \ 0 < x, x'<2kN-1.
\end{align}
\end{proposition}
This formula agrees with formula (2.17) of Theorem 2.6 in \cite{BB23+}. Because of \eqref{deftilphi} we have $\varphi_-(s) = P_{N}^{-1}(s)$
for $s \in \Gamma$, and $\varphi^{-1}(z) = P_{N}(z)$ for $z \in \Gamma' \subset \Omega_{ext}$, allowing us to rewrite the double integral in \eqref{DIntPhi} as 
\begin{align}\label{DIntP}   -\frac{1}{(2\pi i)^2}
\oint_{\Gamma'} \oint_{\Gamma} \left(\prod_{i=1}^{x'} \phi_i(s) \right)^{-1}  
P_{N}^{-1}(s) P_{N}(z) 
\left(\prod_{i=1}^{x} \phi_i(z) \right) \frac{s^{y'}}{z^{y}} 
\frac{ds dz}{z(s-z)}.  
\end{align}

\subsection{The $2\times 2$-periodic Aztec diamond}
\label{2x2Section}
Our final result is an explicit formula for the MVOPs 
$P_N$ and $\widehat{P}_N$ associated with the doubly periodic diamond with periods 2,
or equivalently, in view of the previous results, for the factors in 
the Wiener--Hopf factorization. These formulas are contained in Theorem \ref{theo21} and  are proven in Section~\ref{Sect2x2Const}.

We take $k= 2$ and we assume that the sequences 
$(\vec{\alpha}_{i})_i$, $(\vec{\beta}_{i})_i$, 
$(\vec{\gamma}_{i})_i$ used in \eqref{defWforAD} are periodic
with period $\ell=2$. We write 
\begin{align} \label{phib1} 
	\phi_1(z) & = \phi^b(z;\vec{\alpha}_1,\vec{\gamma}_1)
	= \begin{pmatrix} \gamma_{11} & \alpha_{21} z^{-1} \\
		\alpha_{11} & \gamma_{21} \end{pmatrix}, \\ \label{phig1}
	\phi_2(z) & = \phi^g(z;\vec{\beta}_1)
	= \frac{1}{1-\beta_{11} \beta_{21} z^{-1}} 
	\begin{pmatrix} 1 & \beta_{21} z^{-1} \\
		\beta_{11} & 1 \end{pmatrix}, \\ \label{phib2}
	\phi_3(z) & = \phi^b(z;\vec{\alpha}_2,\vec{\gamma}_2)
	= \begin{pmatrix} \gamma_{12} & \alpha_{22} z^{-1} \\
		\alpha_{12} & \gamma_{22} \end{pmatrix}, \\ \label{phig2}
	\phi_4(z) & = \phi^g(z;\vec{\beta}_2)
	= \frac{1}{1-\beta_{12} \beta_{22}z^{-1}} 
	\begin{pmatrix} 1 & \beta_{22} z^{-1} \\
		\beta_{12} & 1 \end{pmatrix}.
\end{align}
For an integer $N$, the weight matrix \eqref{defWforAD} then is
\begin{equation} \label{defPhi} 
	W = \Phi^N, \qquad \text{where } \quad  \Phi = \phi_1 \phi_2 \phi_3 \phi_4. 
	\end{equation}
Here $\Phi$ is a rational matrix-valued function with poles
at $\beta_{11}\beta_{21}$ and $\beta_{12} \beta_{22}$.

The eigenvalue equation
\begin{equation} \label{EigEq} 
	\det\left(\Phi(z) - \lambda I_2\right) = 0 \end{equation}
is a Harnack curve \cite{KO06}. This in particular implies that for each $z \in \mathbb C \setminus \mathbb R$, 
there are eigenvalues $\lambda_1(z)$, $\lambda_2(z)$ of $\Phi(z)$, with
\begin{equation} \label{Harnack1} 
	|\lambda_1(z)| > |\lambda_2(z)| > 0, \qquad z \in \mathbb C \setminus \mathbb R, \end{equation}
such that $\lambda_1$ and $\lambda_2$ are analytic functions in $\mathbb C \setminus \mathbb R$.

There are four branch points of \eqref{EigEq} on the real line. 
These are those special points $x_j$ for which $\Phi(x_j)$ has
only one distinct eigenvalue. The $x_j$ are real, non-positive 
and can be ordered such that
\begin{equation} \label{xineq}
	-\infty \leq x_3 < x_2 \leq x_1 < x_0 \leq 0, \end{equation}
see \cite[Proposition 2.1]{B21}.
Generically these are strict inequalities. 
The eigenvalues are real and distinct for $z \in (-\infty,x_3) \cup (x_2, x_1) \cup (x_0, \infty)$,
and $\lambda_1$, $\lambda_2$ have analytic continuations to
$\mathbb C \setminus ([x_3,x_2] \cup [x_1,x_0])$,
except for simple poles of $\lambda_1$ at $\beta_{11} \beta_{12}$ and
$\beta_{21} \beta_{22}$ (or a double pole in case these numbers coincide).
We assume that $x_2 < x_1$, which means that the Riemann surface associated
with \eqref{EigEq} has genus one, i.e., it is an elliptic curve. Under this assumption, our main result is a factorization of $P_N$ and $\widehat{P}_N$.
\begin{theorem} \label{theo21} 
	Assume that strict inequality $x_2 < x_1$ holds in \eqref{xineq}. Then
	we have
	\begin{align} \label{PNformula}
		P_N(z) & = C_N E_{-N\omega_0}(z) G^N(z) E^{-1}(z), \\
		\widehat{P}_N(z) & =  E(z) G^N(z) E_{N\omega_0}^{-1}(z) \widehat{C}_N,
		\label{PNhatformula}
	\end{align}
	where $C_N$ and $\widehat{C}_N$ are constant lower triangular invertible matrices,
	and $E$, $G$ and $E_{\pm N \omega_0}$ are matrix-valued functions that will be defined in Definitions \ref{def22}, \ref{def23}, and \ref{def24} below.
\end{theorem}
The proof is rather involved and we give it in Section~\ref{Sect2x2Const}.

Let us comment on the factors appearing in \eqref{PNformula} and \eqref{PNhatformula}.
The matrix-valued function $E$ contains the eigenvectors of $\Phi$ in its columns, and we
use the precise form given in \eqref{defEz}. The matrix-valued function $G$ is a diagonal
matrix with diagonal entries $g_1$ and $g_2$ given in \eqref{defgjz} as a product of ratios
of Jacobi theta functions. Together they can be viewed as coming from a quasi-periodic meromorphic
function on $\mathcal R$ having poles at the two points at infinity (i.e., the poles of $z$), 
and zeros at the two poles of $\lambda$, with $g_j$ for $j=1,2$ being its restriction to the $j$th sheet.

Both $E$ and $G^N$ will have jump discontinuities on the interval $[x_3,x_0]$ on the real line,
see \eqref{Ejump} and \eqref{Gjump} below. The
factors $E_{-N\omega_0}$ in \eqref{PNformula} and $E_{N\omega_0}^{-1}$ in \eqref{PNhatformula} are carefully chosen to eliminate these jumps, so that \eqref{PNformula} and \eqref{PNhatformula} are entire, and
in fact polynomials. For each real $\omega$ we construct $E_{\omega}(z)$ in \eqref{defEomega}
as a modification of $E(z)$ by ratios of Jacobi theta functions with poles at the zeros of the appropriate entries of $E$ to avoid introducing unwanted poles. The parameter $\omega$ appears as a shift in the Jacobi theta
functions in the numerator. The choice $\omega = \pm N\omega_0$, with $\omega_0$ given in \eqref{defomega}, will precisely compensate for the jumps
in $G^N$ and $E$. For full details we refer to the proof in Section~\ref{Sect2x2Const}.

Finally, $C_N$ and $\widehat{C}_N$ are normalizing factors to make 
\eqref{PNformula} and \eqref{PNhatformula}
monic. 

\begin{remark}
    Due to the relation between Wiener--Hopf factorizations and the domino shuffle noted in \cite{CD23}, it should be possible to use formulas \eqref{PNformula}, \eqref{PNhatformula} together with the results of Section \ref{WH_MVOPs} to obtain explicit solutions for the domino shuffle weight dynamics with generic $2\times2$-periodic initial data in term of Jacobi theta functions. We will however not pursue this goal in the present paper.   
\end{remark}


\section{Proofs of Propositions \ref{prop1}, \ref{prop2}, and \ref{prop3}}
\label{ProofsProp2}
\subsection{Proof of Proposition \ref{prop1}, parts (a), (b), (c)} \label{proofprop1abc}

\begin{proof} 
	(a)
	The poles of $W^{-1}$ are in the exterior of $\Gamma$, and therefore $P_N$ defined by \eqref{defP}
	is analytic away from $\Gamma$. For $z \in \Gamma$,
	we have 
	\[ \lim_{w \to z, \, w \in \Omega_{int}} P_N(w) = \lim_{w \to z, \, w \in \Omega_{ext}} P_N(w) \] 
	due to the definition \eqref{defP} and the second identity in \eqref{Wfactors}. 
	Thus $P_N$ has an analytic continuation across $\Gamma$ (also denoted by $P_N$)
	and $P_N$ is an entire matrix-valued function.
	Because of \eqref{phiasymp} and \eqref{defP} we have $z^{-N} P_N(z) \to I_k$
	as $z \to \infty$. Thus $P_N$ is a monic matrix-valued polynomial of degree $N$.
	
	From \eqref{defP} we also get that
	\begin{equation} \label{defPW}  
		P_N(z) W(z) = \begin{cases} \varphi(z)^{-1} W(z), & \quad z \in \Omega_{ext}, \\
			\varphi(z), & \quad z \in \Omega_{int}. 
	\end{cases} \end{equation}
	The poles of $W$ are in $\Omega_{int}$, and we see from \eqref{defPW} that $P_NW$ is also entire.
	Since $P_N$ is a polynomial and $W$ is rational, it follows that $P_N W$ is a matrix-valued polynomial. This proves part (a).
	\medskip
	
	(b) is proved in the same way. 
	
	\medskip
	(c)  The factors in the Wiener--Hopf factorization \eqref{Wfactors} 
	are invertible matrices, and therefore $f$ given by
 \begin{align*}
 f := \begin{cases}
     \cfrac{\det \widehat{\varphi}}{\det \varphi}  &\text{ in }  \Omega_{ext},
     \\[2ex]
     \cfrac{\det \varphi}{\det \widehat{\varphi}} &\text{ in }  \Omega_{int},
 \end{cases}
 \end{align*}
	is well-defined and analytic in $\mathbb C \setminus \Gamma$. From the
	factorization \eqref{Wfactors} it follows that $f_+ = f_-$ on $\Gamma$,
	hence $f$ is entire. From the asymptotic behavior \eqref{phiasymp} it follows that $f(z) \to 1$
	as $z \to \infty$, and we conclude that $f \equiv 1$ by Liouville's theorem.
	Thus 
	\[ \det \varphi(z) = \det \widehat{\varphi}(z), \qquad z \in \mathbb C \setminus \Gamma. \]
	Then it is clear from \eqref{defP} and \eqref{defhatP} that $\det P_N(z) = \det \widehat{P}_N(z)$ 	for every $z \in \mathbb C$.
	From \eqref{defP} it also follows that $\det P_N$ has no zeros in $\Omega_{ext} \cup \Gamma$,
	while in $\Omega_{int}$ we have $\det P_N = \frac{\det \varphi}{\det W}$. 
	Since $\det \varphi$ has no zeros, $\det P_N$ has zeros at
	the poles of $\det W$ only. Since $\det P_N$ is a polynomial of degree $kN$, 
	$\det W$ will have $kN$ poles, with poles counted according to their multiplicities.
\end{proof}

\subsection{Proof of Proposition \ref{prop1}(d)} \label{proofprop1d}

For the proof of part (d) of Proposition \ref{prop1} we need a basic fact about
the solution of the RH problem from Section \ref{MVOPs_WH}. It is well-known, but we could not find a good reference for it.	

\begin{lemma} \label{uniqueMVOP}
	The RH problem \eqref{RHYjump}-\eqref{RHYasymp} has a solution if and only
	if the monic matrix-valued polynomial $P_N$ of degree $N$ satisfying \eqref{MVOP1}
	uniquely exists, and in that case 
	\begin{equation} \label{RHYPN} 
		P_{N} = \begin{pmatrix} I_k & 0_k \end{pmatrix}  Y \begin{pmatrix} I_k \\ 0_k  \end{pmatrix}. 
	\end{equation}
\end{lemma}
\begin{proof}
	If the RH problem $Y$ has a solution then it is necessarily of the form \eqref{RHYsol}
	and $P_N$ as in \eqref{RHYPN} is a monic MVOP of degree $N$. If $\widetilde{P}_N$ would be another
	monic MVOP of degree $N$, then we could define $\widetilde{Y}$ as in \eqref{RHYsol}
	but with $P_N$ replaced by $\widetilde{P}_N$. It would lead to a different solution of the RH problem.
	This is a contradiction since the solution of the RH problem (if it exists) is unique.
	
	Conversely, suppose that $P_N$ uniquely exists.
	Let $\mathcal V$ be the vector space of matrix-valued polynomials $Q$ of degree $\leq N-1$
	such that
	\[ \frac{1}{2\pi i} \oint_{\Gamma} Q(z) W(z) z^j dz = 0_k, \qquad j=0, 1, \ldots, N-2. \] 
	Recall that the matrices have size $k \times k$, and so there are $k^2 \times (N-1)$
	homogeneous conditions for the $k^2 \times N$ degrees of freedom in $Q$. Thus
	$\dim \mathcal V \geq k^2$.
	Suppose $\widetilde Q \in \mathcal V$ is in the kernel of the linear map
	\[ L: \quad \mathcal V \to \mathbb C^{k \times k} : \quad Q \mapsto \frac{1}{2\pi i} \oint_{\Gamma} Q(z) W(z) z^{N-1} dz. \]
	Then $P_N + \widetilde Q$ is a monic matrix polynomial of degree $N$ satisfying the
	matrix-valued orthogonality \eqref{MVOP1}, and thus by uniqueness of $P_N$, it follows
	that $\widetilde Q = 0_k$.  Hence $L$ has a trivial kernel. Since $\dim \mathcal V \geq k^2 =  \dim \mathbb C^{k \times k}$, it then follows that the dimensions are equal and that $L$ is an isomorphism.
	In particular, $L$ is surjective, and there exists $Q_{N-1} \in \mathcal V$ with $L Q_{N-1} = -I_k$. Then $Q_{N-1}$ satisfies
	the conditions \eqref{QNortho}, and we can use $P_N$ and $Q_{N-1}$ to build $Y$ according
	to \eqref{RHYsol} and $Y$ will solve the RH problem.   
\end{proof}

\begin{proof}[Proof of Proposition \ref{prop1}(d)]
	The matrix-valued orthogonalities \eqref{MVOP1} and \eqref{MVOP3} 
	follow from parts (a) and (b) of Proposition \ref{prop1}. We have to show the uniqueness.
	
	By Lemma~\ref{uniqueMVOP}, there exists a unique  matrix-valued polynomial of degree $N$ satisfying \eqref{MVOP1}
	if and only if the RH problem \eqref{RHYjump}-\eqref{RHYasymp} has a solution.
	The solution has to be of the form \eqref{RHYsol}, and so we have to show that 
	a matrix-valued polynomial $Q_{N-1}$ of degree $\leq N-1$ exists such that 
	\eqref{QNortho} holds. 
	Having $P_N$ such that $P_NW$ is a matrix-valued polynomial, we look
	for $Q_{N-1}$ of the form 
		\begin{equation}\label{def:QwithH}
		Q_{N-1}(z) = \begin{cases}
			\left( \ds \frac{1}{2\pi i} \oint_{\Gamma}
			H_+(s)P_N^{-1}(s)  \frac{ds}{s-z} \right) P_N(z),
			& z \in \Omega_{ext}, \\
			\left( \ds \frac{1}{2\pi i} \oint_{\Gamma}
			H_+(s) P_N^{-1}(s)  \frac{ds}{s-z} \right) P_N(z)
			- H(z),
			& z \in \Omega_{int},
		\end{cases}
	\end{equation}
	where $H$ is an analytic matrix-valued function in $\Omega_{int}$
	with continuous boundary values $H_+$ on the boundary 
	$\partial \Omega_{int} = \Gamma$. 
	We know from part (c) that $\det P_N(s) \neq 0$ for $s \in \Gamma$,
	and so \eqref{def:QwithH} is well-defined and analytic in $\mathbb C \setminus \Gamma$.
	The Sokhotski--Plemelj formula tells us that $Q_{N-1,+} = Q_{N-1,-}$ on $\Gamma$. Hence $Q_{N-1}$
	is entire. From \eqref{QNformula3} we see that $Q_{N-1}(z) = O(z^{N-1})$ 
	as $z \to \infty$, and we conclude that $Q_{N-1}$ is a matrix-valued polynomial of degree $\leq N-1$. 
	
	We want to determine $H$ such that \eqref{QNortho} is satisfied.
	We observe from \eqref{def:QwithH} that $Q_{N-1}(z) W(z) z^j$ for $z \in \Omega_{int}$
	has two terms. Since $P_NW$ is entire, the first term
	\[ \left( \ds \frac{1}{2\pi i} \oint_{\Gamma}
	H_+(s) P_N^{-1}(s)  \frac{ds}{s-z} \right) P_N(z) W(z) z^j \]
	is analytic in $\Omega_{int}$ for every $j=0,1,2, \ldots$, and it is continuous
	up to the boundary. By Cauchy's theorem
	its integral over $\Gamma$ vanishes. Hence for $j=0,1,2,\ldots$,
	\begin{equation} \label{QNorthogonality1} 
		\frac{1}{2\pi i} \oint_{\Gamma} Q_{N-1}(z) W(z) z^j dz = 
	-\frac{1}{2\pi i} \oint_{\Gamma} H_+(z) W(z) z^j dz. \end{equation}
	Now we can verify that the required $H$ is provided by the Wiener--Hopf factorization
	$W = \widehat{\varphi}_+ \widehat{\varphi}_-$ on $\Gamma$.
	Indeed, putting 
	\begin{equation} \label{defH} 
		H(z) =  \widehat{\varphi}^{-1}(z),  \quad z \in \Omega_{int}, \end{equation}
	we have $H_+ W = \widehat{\varphi}_-$ on $\Gamma$, and 
	\begin{equation} \label{QNorthogonality2} \frac{1}{2\pi i} \oint_{\Gamma} H_+(z) W(z) z^j dz 
	= \frac{1}{2\pi i} \oint_{\Gamma} \widehat{\varphi}_-(z) z^j dz. 
	\end{equation}
	Since $\widehat{\varphi}$ is analytic in $\Omega_{ext}$,
	the only contribution to the integral on the right-hand side
	comes from the residue at infinity.
	Since $\widehat{\varphi}(z) = z^{-N} I_k + O\left(z^{-N-1}\right)$ as $z \to \infty$, we find
	\begin{equation} \label{QNorthogonality3}	
		\frac{1}{2\pi i} \oint_{\Gamma} \widehat{\varphi}_-(z) z^j dz
		= \begin{cases} 0_k, &  j=0,1, \ldots, N-2, \\
			I_k, & j=N-1.
			\end{cases} \end{equation}
	Combining \eqref{QNorthogonality1}, \eqref{QNorthogonality2}, \eqref{QNorthogonality3}
	we obtain the identities \eqref{QNortho}. Thus $Q_{N-1}$ exists
	and is given by \eqref{def:QwithH} with $H$ defined by \eqref{defH}.
	Writing $P_N$ in terms of $\varphi$ as in \eqref{defP} we also obtain
	\eqref{QNformula2}.
\end{proof}

\subsection{Proof of Proposition \ref{prop2}} \label{proofprop2}

\begin{proof}
	(a) Since $P_N$ uniquely exists, the RH problem \eqref{RHYjump} and \eqref{RHYasymp}
	has a solution $Y$, see Lemma \ref{uniqueMVOP}, which is necessarily unique and it is given by \eqref{RHYsol}.
	It is well-known that $\det Y(z) = 1$ for $z \in \mathbb C \setminus \Gamma$,
	see e.g.\ \cite{BFM21, GIM11},
	and hence the inverse $Y^{-1}(z)$ exists for $z \in \mathbb C \setminus \Gamma$.
	Consider 
	\begin{equation} \label{Xdef} 
		X(z) = \begin{pmatrix} 0_k & -I_k \\ I_k & 0_k \end{pmatrix} Y^{-t}(z) 
		\begin{pmatrix} 0_k & I_k \\ -I_k & 0_k \end{pmatrix}, \qquad z \in \mathbb C \setminus \Gamma, \end{equation}
	where $Y^{-t} = (Y^{-1})^t = (Y^t)^{-1}$ is the inverse transpose of $Y$.
	Then it is immediate from \eqref{RHYjump} and \eqref{RHYasymp}
	that
	\begin{equation} \label{RHXjump}
		X_+ = X_- \begin{pmatrix} I_k & W^t \\ 0_k & I_k \end{pmatrix} \quad
		\text{ on } \Gamma, \end{equation}
	and
	\begin{equation} 
		X(z) = \left( I_{2k} + O(z^{-1})\right) \begin{pmatrix} z^N I_k & 0_k \\
			0_k & z^{-N} I_k \end{pmatrix} \text{ as } z \to \infty. 
	\end{equation} 
	This is the RH problem for matrix-valued orthogonal polynomials
	with respect to the weight matrix $W^t$. The left upper block of $X$
	is thus a monic matrix-valued polynomial $P$ of degree $N$ with orthogonality with respect to $W^t$, i.e.,
	\[ \frac{1}{2\pi i} \oint_{\Gamma} P(z) W^t(z) z^j dz = 0_k, \qquad j=0,1, \ldots, N-1. \]	
	Taking the transpose of $P$ and comparing with \eqref{MVOP3} we find that $\widehat{P}_N$ 	exists and
	\begin{align*} \widehat{P}_N & = P^t = 
		\left[ \begin{pmatrix} I_k & 0_k \end{pmatrix} X \begin{pmatrix} I_k \\ 0_k \end{pmatrix} \right]^t. \end{align*}
	Then using \eqref{Xdef} we find  \eqref{RHYhatP}. 
	
	The relation \eqref{Xdef} between the RH problem $Y$ for $P_N$ and
	the RH problem $X$ for $\widehat{P}_N$ was already found by 
	Branquinho et al. \cite[Theorem 3]{BFM21}. 
	
\medskip
(b)  
Suppose $P_NW$ is a matrix-valued polynomial. Then for $z \in \Omega_{ext}$ we have
$\frac{1}{2\pi i} \oint_{\Gamma} \frac{P_N(s) W(s)}{s-z} ds = 0_k$ by Cauchy's theorem.
Hence by \eqref{RHYsol}
\begin{equation} \label{RHYsol2} 
	Y(z) = \begin{pmatrix} P_N(z) & 0_k \\ Q_{N-1}(z) &  \ast  \end{pmatrix} \quad z \in \Omega_{ext}, \end{equation}
	where $\ast$ denotes an entry  whose precise form is not important at the moment.
Since $Y(z)$ is invertible for every $z \in \mathbb C \setminus \Gamma$ it follows
from \eqref{RHYsol2} that $P_N(z)$ is invertible for $z \in \Omega_{ext}$.
By part (a) we have
\[ Y^{-1}(z) = \begin{pmatrix} \ast & \ast \\ \ast & \widehat{P}_N(z) \end{pmatrix},
	\quad z \in \mathbb C \setminus \Gamma. \]
Comparing with \eqref{RHYsol2} we conclude that 
\begin{equation} \label{RHYsol3} 
	Y(z) = \begin{pmatrix} P_N(z) & 0_k \\ Q_{N-1}(z) & \widehat{P}_N^{-1}(z) \end{pmatrix}, \quad z \in \Omega_{ext}. \end{equation}
and in particular $\widehat{P}_N(z)$ is invertible for $z \in \Omega_{ext}$.
This proves part (b).

\medskip

(c) Because of part (b), the matrix-valued functions $\varphi$ and $\widehat{\varphi}$  are
well-defined by the formulas \eqref{defphi} and \eqref{deftilphi}, and they are both
analytic in $\mathbb C \setminus \Gamma$.
By construction the factorizations \eqref{Wfactors} hold, and also the
asymptotic behavior \eqref{phiasymp} is immediate since both $P_N$ and $\widehat{P}_N$
are monic polynomials of degree $N$. Thus part (c) follows.
\end{proof}

Comparing \eqref{RHYsol} and \eqref{RHYsol3} we see that, in the situation of Proposition~\ref{prop2}
\begin{equation} \widehat{\varphi}(z) = \widehat{P}_N^{-1}(z) =
	\frac{1}{2\pi i} \oint_{\Gamma} \frac{Q_{N-1}(s) W(s)}{s-z} ds,
		\qquad z \in \Omega_{ext},
\end{equation}
which is a somewhat remarkable identity.

\subsection{Proof of Proposition \ref{prop3}} \label{proofprop3}

\begin{proof}[Proof of Proposition \ref{prop3}] (a) From \eqref{RHYsol3} and the fact that $\det Y \equiv 1$ it follows that
\[ \det P_N(z) = \det \widehat{P}_N(z) \]
for every $z \in \Omega_{ext}$, and hence for every $z \in \mathbb C$, since these
are polynomials. Since $\det\left(P_NW\right)$ has no zeros in $\Gamma \cup \Omega_{int}$,
the same is true for $\det\left(W \widehat{P}_N\right)$.

\medskip

(b) It is clear from the definitions
\eqref{defphi} and \eqref{deftilphi} that $\varphi(z)$ and  $\widehat{\varphi}(z)$ are
invertible for $z \in \Omega_{ext}$. Also the limiting values $\varphi_{\pm}$ and $\widehat{\varphi}_{\pm}$ 
are invertible on $\Gamma$, because of the factorization \eqref{Wfactors} 
(which was proved in Proposition \ref{prop2}(c)),
and the fact that it is assumed now that  $\det \left(P_N W\right) \neq 0$ on $\Gamma$, hence $\det W \neq 0$
on $\Gamma$. 
Since $\det \left(P_NW(z)\right)  \neq 0$ for $z \in \Omega_{int}$,
it follows from \eqref{deftilphi} that $\varphi(z)$ is also invertible for $z \in \Omega_{int}$.
Similarly, $\widehat{\varphi}(z)$ is invertible for $z \in \Omega_{int}$ by \eqref{defphi} and 
$\det\left(W \widehat{P}_N(z)\right)
\neq 0$ for $z \in \Omega_{int}$, see part (a).

\medskip
(c) From \eqref{defphi} we obtain $W^{-1} = \widehat{P}_N \widehat{\varphi}^{-1}$
in $\Omega_{int}$. Hence there are no poles of $W^{-1}$ in $\Omega_{int}$. This implies
that the conditions of Proposition \ref{prop1} are satisfied. We proved
formula \eqref{QNformula3} in the proof of Proposition \ref{prop1}(d)
and so it is also valid in the present situation. 

\medskip

(d) Inverting \eqref{RHYsol3} we have
	\[ Y^{-1}(w) = \begin{pmatrix} P_N^{-1}(w) & 0_k \\ - \widehat{P}_N(w) Q_{N-1}(w) P_N^{-1} (w) & \widehat{P}_N(w)  \end{pmatrix},
	\quad w \in \Omega_{ext}. \] 
	Hence for $w,z \in \Omega_{ext}$,
	\begin{align*} (z-w) R_N(w,z) 
		& = \begin{pmatrix}  - \widehat{P}_N(w) Q_{N-1}(w) P_N^{-1} (w) & \widehat{P}_N(w) \end{pmatrix}
		\begin{pmatrix} P_N(z) \\ Q_{N-1}(z) \end{pmatrix} \\
		& =  \widehat{P}_N(w) Q_{N-1}(z)  - \widehat{P}_N(w) Q_{N-1}(w) P_N^{-1} (w) P_N(z). 
	\end{align*}
	Inserting the integral representation \eqref{QNformula3} we obtain from this
	\begin{align*} 
 (z-w)&R_N(w,z) \\
		&=
		\widehat{P}_N(w) \left[ \frac{1}{2\pi i} \oint_{\Gamma} \widehat{P}_N^{-1}(s) W^{-1}(s) P_N^{-1}(s)
		\frac{ds}{s-z} \right.
  \\&\hspace{3cm} \left. - \frac{1}{2\pi i} \oint_{\Gamma} \widehat{P}_N^{-1}(s) W^{-1}(s) P_N^{-1}(s)
		\frac{ds}{s-w}   \right] P_N(z) \\
		&= 	\widehat{P}_N(w) \left[ \frac{1}{2\pi i} \oint_{\Gamma} \widehat{P}_N^{-1}(s) W^{-1}(s) P_N^{-1}(s)
		\left(\frac{1}{s-z} - \frac{1}{s-w} \right) ds \right] P_N(z) \\
		&= 	\widehat{P}_N(w) \left[ \frac{1}{2\pi i} \oint_{\Gamma} \widehat{P}_N^{-1}(s) W^{-1}(s) P_N^{-1}(s)
		\frac{z-w}{(s-z)(s-w)}   ds \right] P_N(z).
	\end{align*}
\end{proof}

\section{Proof of Theorem \ref{theo21}} \label{Sect2x2Const}
The formulas \eqref{PNformula} and \eqref{PNhatformula} can be found from an analysis of the RH problem similar to what was done in \cite{DK21}. The  RH problem is exactly solvable in terms of Jacobi theta functions and does not require the typical Deift--Zhou nonlinear steepest descent analysis. In fact, the main idea of the RH analysis, the diagonalization of the weight matrix $W$, cf.~\cite{DK21, GK21}, suffices to give a direct more compact proof of Theorem \ref{theo21} circumventing the RH analysis entirely. This is the approach we will take in the following.

\subsection{Preliminaries}
\label{Sect_Prelimianries}

\subsubsection{Riemann surface $\mathcal R$}\label{Sect:RiemannSurface}
The property \eqref{Harnack1} introduces a natural sheet structure on 
the compact Riemann surface $\mathcal R$ associated with the algebraic equation \eqref{EigEq}.
The two sheets are copies of $\mathbb C \setminus ([x_3,x_2] \cup [x_1, x_0])$
that are glued together along the cuts $[x_3,x_2]$ and $[x_1,x_0]$ in the
usual crosswise manner. 
A point $(z,\lambda_1(z))$ is on the first sheet, while $(z,\lambda_2(z))$ is
on the second sheet. We also use the notation $z^{(j)}$ to denote the point
on the $j$th sheet of the Riemann surface that projects to $z$.

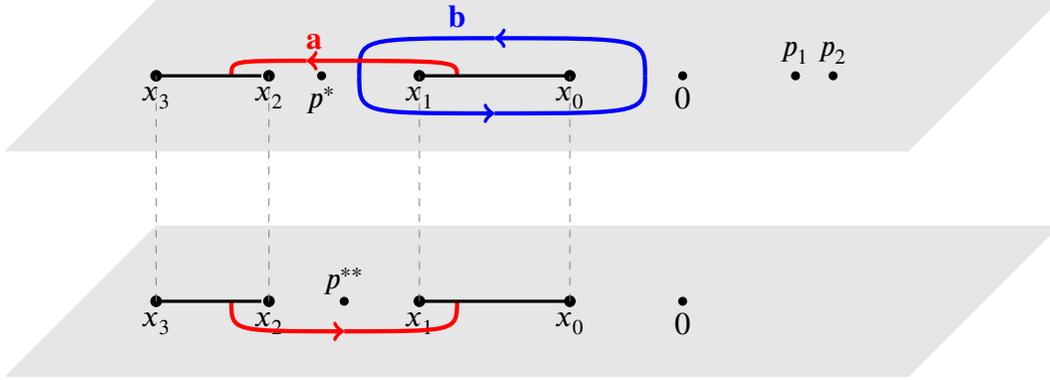
\begin{figure}[t]
	\begin{center}
		\begin{tikzpicture}(15,10)(0,0)
			\filldraw[gray!20!white] (5,0) --++(12,0) --++(2,2) --++(-12,0) --++(-2,-2);
			
			\filldraw[gray!20!white] (5,3) --++(12,0) --++(2,2) --++(-12,0) --++(-2,-2);
			
			\filldraw (7,1)  circle(2pt);
			\filldraw[black] (9.5,1) circle(1.5pt);	
			\filldraw[black] (9.2,4) circle(1.5pt);	
			\draw[black]   (9.5,1.6) node[below] {$p^{**}$};
			\draw[black]   (9.2,4) node[below] {$p^*$}; 
			\filldraw (8.5,1)  circle (2pt);	 
			\filldraw (10.5,1)  circle (2pt);	 
			\filldraw (12.5,1)  circle (2pt);	 
			\draw   (8.5,1) node[below] {$x_2$};
			\draw   (10.5,1) node[below] {$x_1$};
			\draw   (12.5,1) node[below] {$x_0$};
			\draw   (7,1) node[below] {$x_3$};
			
			\filldraw[black] (15.5,4) circle(1.5pt);	
			\filldraw[black] (16,4) circle(1.5pt);	
			\draw[black]   (15.5,4.6) node[below] {$p_1$};
			\draw[black]   (16,4.6) node[below] {$p_2$}; 
		
			\draw[blue, ultra thick] (9.7,4) .. controls (9.7,4.5).. (11.5,4.5);
			\draw[blue, ultra thick,->] (13.5,4) .. controls (13.5,4.5).. (11.5,4.5);
			\draw[blue, ultra thick,->] (9.7,4) .. controls (9.7,3.5).. (11.5,3.5);
			\draw[blue, ultra thick] (13.5,4) .. controls (13.5,3.5).. (11.5,3.5);
		
			\draw[red, ultra thick] (8,4) .. controls (8,4.2) .. (9,4.2);
			\draw[red, ultra thick,->] (11,4) .. controls (11,4.2) .. (9,4.2);
			\draw[red, ultra thick,->] (8,1) .. controls (8,0.6) .. (9.5,0.6);
			\draw[red, ultra thick] (11,1) .. controls (11,0.6) .. (9.5,0.6);
			\draw[red]   (9.1,4.2) node[above] {$\textbf{a}$};
			\draw[blue]   (11,4.5) node[above] {$\textbf{b}$};

			\draw[very thick,black] (7,1)--++(1.4,0); 
			\draw[very thick,black] (10.5,1)--++(2,0); 
			\filldraw (7,4)  circle (2pt);
			\filldraw (8.5,4)  circle (2pt);	 
			\filldraw (10.5,4)  circle (2pt);	 
			\filldraw (12.5,4)  circle (2pt);

			\filldraw (14,4)  circle (1.5pt);
			\filldraw (14,1)  circle (1.5pt);
			\draw   (14,4) node[below] {$0$};
			\draw   (14,1) node[below] {$0$};
			
			\draw   (8.5,4) node[below] {$x_2$};
			\draw   (10.5,4) node[below] {$x_1$};
			\draw   (12.5,4) node[below] {$x_0$};
			\draw   (7,4) node[below] {$x_3$};
			
			\draw[very thick, black] (7,4)--++(1.4,0); 
			\draw[very thick,black] (10.5,4)--++(2,0); 
			
			\draw[dashed,help lines] (7,1)--(7,4);
			\draw[dashed,help lines] (8.5,1)--(8.5,4);	
			\draw[dashed,help lines] (10.5,1)--(10.5,4);
			\draw[dashed,help lines] (12.5,1)--(12.5,4);
		\end{tikzpicture}
	\end{center}
	\caption{The two-sheeted Riemann surface $\mathcal R$ with branch cuts along $[x_3,x_2]$ and $[x_1,x_0]$,
		and with $\textbf{a}$ and $\textbf{b}$ cycles. 
		The Riemann surface has special points $p^*$ and $p^{**}$ on the bounded oval, 
		and $p_1$ and $p_2$ on the unbounded oval. 
		The points $p^*$ and $p^{**}$ are given by \eqref{pstar} and they could be anywhere on the bounded oval. 
		The points $p_1$ and $p_2$ given by \eqref{polelambda} are the poles of $\lambda$ and they lie
		on the positive real axis of the first sheet. The zeros of $\lambda$ (not shown in the figure) are on the 
		positive real axis of the second sheet.\label{RSurface}}
	\end{figure}

The real part of $\mathcal R$ has two connected components. There
is an unbounded component containing all points where either $z$ or $\lambda$
is equal to $0$ or $\infty$. This is called the unbounded oval.
The other component is bounded. We call it the bounded oval, and it
consists of points $(z,\lambda)$ on the Riemann surface with $z \in [x_2,x_1]$.
As we have assumed that $x_2 < x_1$, the Riemann surface has genus one.

Being of genus one, the Riemann surface has a unique holomorphic differential $\eta$ 
with $\oint_{\textbf{a}} \eta = 1$, where $\textbf{a}$ is homotopic to the bounded oval viewed
as a cycle on the Riemann surface with orientation from $x_1$ to $x_2$ on the first sheet,
and from $x_2$ to $x_1$ on the second sheet, see also Figure \ref{RSurface}.
Explicitly we have (in case $x_{3} > -\infty$)
\begin{equation} 
	\eta =  \frac{C^{-1} dz}{\left[(z-x_{3})(z-x_2)(z-x_1)(z-x_0)\right]^{1/2}} \end{equation}
with positive constant 
$ C = 2 \int_{x_2}^{x_1} \frac{dx}{\sqrt{(x-x_3)(x-x_2)(x-x_1)(x-x_0)}}$,
where we take the positive square root for every $x \in [x_2,x_1]$. We also define the cycle $\bf{b}$ to go around the interval $[x_1, x_0]$ on the first sheet, see Figure~\ref{RSurface}.

We define
\begin{equation} \label{deftau} 
	\tau = \oint_{\bf{b}} \eta \in i \mathbb R^+. \end{equation}
Then $\mathcal R$ is conformally equivalent to the complex torus $\mathbb C \slash (\mathbb Z + \tau \mathbb Z)$
under the Abel map. We choose the Abel map with base point at $p_0 = (x_0, \lambda(x_0))$
(note $\lambda_1(x_0) = \lambda_2(x_0)$),
\begin{equation} \label{Abelmap} 
	\mathcal A :  p \in \mathcal R \mapsto \int_{p_{0}}^p \eta \in \mathbb C \slash
	(\mathbb Z + \tau \mathbb Z).
\end{equation}
We restrict the Abel map to $\mathcal R' := \{ p \in \mathcal R \mid z(p) \not\in [x_3, x_1] \}$ and we choose a path of integration in \eqref{Abelmap} from $p_0$ to $p$ that lies
within $\mathcal R'$. Then $\mathcal A$ maps $\mathcal R'$ to the open rectangle
\begin{equation} \label{rectangle} 
	\{ u = x + y \tau  \mathbb  \mid -\tfrac{1}{2} < x < \tfrac{1}{2}, \, - \tfrac{1}{2} < y < \tfrac{1}{2} \},
\end{equation} 
see Figure~\ref{figjv}.

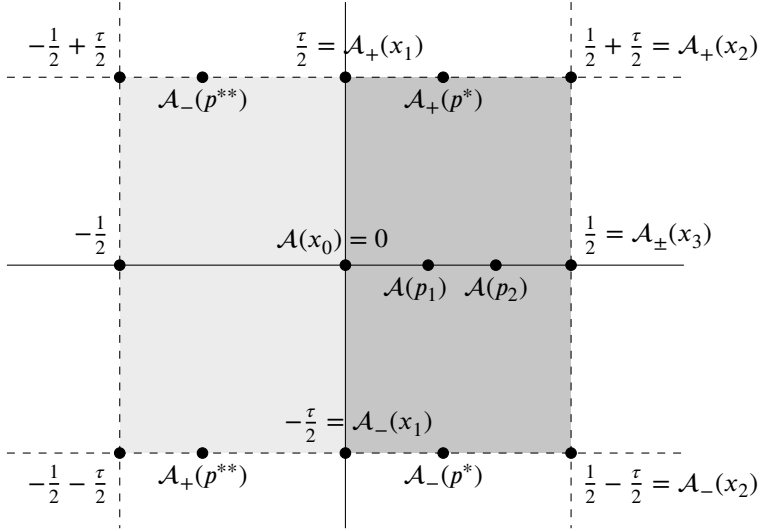
\begin{figure}[t]
\centering
\begin{tikzpicture}
\path [fill=gray!45] plot (0,-2.5) -- (3,-2.5) -- (3,2.5) -- (0,2.5) --cycle;
\path [fill=gray!15] plot (-3,-2.5) -- (0,-2.5) -- (0,2.5) -- (-3,2.5) --cycle;

\draw (-4.5,0) -- (4.5,0);
\draw (0,-3.5) -- (0,3.5);

\draw[dashed] (-4.5,2.5) -- (4.5,2.5);
\draw[dashed] (-4.5,-2.5) -- (4.5,-2.5);
\draw[dashed] (-3,3.5) -- (-3,-3.5);
\draw[dashed] (3,3.5) -- (3,-3.5);

\draw[fill] (0,0) circle (0.07) node[above] {\footnotesize{$\mathcal A(x_0)=0 \quad$}};
\draw[fill] (1.1,0) circle (0.07) node[below] {\footnotesize{$\mathcal A(p_1) \quad$}};
\draw[fill] (1.3,2.5) circle (0.07) node[below] {\footnotesize{$\mathcal A_+(p^{*})$}};
\draw[fill] (-1.9,2.5) circle (0.07) node[below] {\footnotesize{$\mathcal A_-(p^{**})$}};
\draw[fill] (1.3,-2.5) circle (0.07) node[below] {\footnotesize{$\mathcal A_-(p^{*})$}};
\draw[fill] (-1.9,-2.5) circle (0.07) node[below] {\footnotesize{$\mathcal A_+(p^{**})$}};
\draw[fill] (2,0) circle (0.07) node[below] {\footnotesize{$\mathcal A(p_2)$}};

\draw[fill] (-3,0) circle (0.07) node[above left] {\footnotesize{$-\frac{1}{2}$}};
\draw[fill] (3,0) circle (0.07) node[above right] {\footnotesize{$\frac{1}{2} = \mathcal A_{\pm}(x_3)$}};
\draw[fill] (0,2.5) circle (0.07) node[above] {\footnotesize{$\quad \frac{\tau}{2} = \mathcal A_+(x_1)$}};
\draw[fill] (0,-2.5) circle (0.07) node[above] {\footnotesize{$\quad - \frac{\tau}{2} = \mathcal A_-(x_1) $}};
\draw[fill] (-3,2.5) circle (0.07) node[above left] {\footnotesize{$-\frac{1}{2} + \frac{\tau}{2}$}};
\draw[fill] (3,2.5) circle (0.07) node[above right] {\footnotesize{$\frac{1}{2} + \frac{\tau}{2} = \mathcal A_+(x_2)$}};
\draw[fill] (-3,-2.5) circle (0.07) node[below left] {\footnotesize{$- \frac{1}{2} -\frac{\tau}{2}$}};
\draw[fill] (3,-2.5) circle (0.07) node[below right] {\footnotesize{$\frac{1}{2} - \frac{\tau}{2} = \mathcal A_-(x_2)$}};

\end{tikzpicture}
\caption{Image of $\mathcal R' = \{ p \in \mathcal R \mid z(p) \not\in [x_3,x_1]\}$ 
	under the Abel map \eqref{Abelmap}. The region in dark (light) gray denotes the image 
	of the upper (lower) sheet.	As in Figure \ref{RSurface}, the point $p^*$ ($p^{**}$) is 
	assumed to be on the first (second) sheet of the bounded oval. 
	In the figure, the branch points $x_1, x_2, x_3$ are considered to be on the first sheet. So for example $\mathcal A_+(x_2)$
	denotes the limit of $\mathcal A(q)$ as $q$ tends to $x_2$ with $q$ in the upper half plane of the first sheet. \label{figjv}
	}
\end{figure}

The unbounded oval is mapped to the real interval $(-\frac{1}{2}, \frac{1}{2})$ 
and for $p$ in the bounded oval we
use $\mathcal A_+(p)$ to denote the limit of $\mathcal A(q)$ as $\mathcal R' \ni q \to p$ from the
upper half plane, i.e., the limit is taken from the upper half plane on the $j$th sheet if $p$ is on the $j$th sheet,
for $j=1,2$. Similarly $\mathcal A_-(p)$ is the limit from the lower half plane.
Then $\Im \mathcal A_\pm(p)  = \pm \frac{\Im \tau}{2}$ if the point $p$ on the bounded oval is on the first sheet
and $\Im \mathcal A_\pm(p)  = \mp \frac{\Im \tau}{2}$ if $p$ is on the second  sheet. Consequently, for a point $p$ on the bounded oval we have
\begin{equation} \label{Abeljump}
	\mathcal A_+(p) = \mathcal A_-(p) + \begin{cases}
		\tau & \text{ if } p  \text{ is on the first sheet,} \\
		- \tau & \text{ if } p  \text{ is on the second sheet}.
\end{cases} \end{equation}

The Jacobi theta function (we follow the notation of Akhiezer \cite{Akhiezer}),
\begin{equation} \label{Jacobi1} 
	\vartheta(u; \tau) = \sum_{n=-\infty}^{\infty} \exp \left( \pi i n^2 \tau + 2\pi i nu \right), \quad u \in \mathbb C \end{equation}
is an entire function with simple zeros at  
\begin{equation} \label{defK} K  + \left(\mathbb Z + \tau \mathbb Z \right), \qquad	
	K = \frac{1+\tau}{2}, \end{equation}
and no other zeros. Moreover,  the quasi-periodicity properties
\begin{equation} \label{Jacobi2} 
	\vartheta(u+1;\tau) = \vartheta(u;\tau), \qquad 
	\vartheta(u+\tau;\tau) = \exp \left(-\pi i (\tau + 2u) \right) \vartheta(u;\tau) \end{equation}
hold as can be verified by a direct computation. With $\tau$ being understood, we simply write $\vartheta(u)$.	

\subsubsection{Matrix-valued function $E$}

The matrix-valued function $E$ contains the eigenvectors of $\Phi$. Recall that $\Phi$ is given by \eqref{defPhi}.

\begin{definition} \label{def22} We define
	\begin{multline} \label{defEz} 
		E(z) = (z-\beta_{11}\beta_{12})(z-\beta_{21}\beta_{22})
		\begin{pmatrix} \Phi_{12}(z) & \Phi_{12}(z) \\
			\lambda_1(z) - \Phi_{11}(z) & \lambda_{2}(z)-\Phi_{11}(z) \end{pmatrix}, \\
			z \in \mathbb C \setminus ([x_3,x_2] \cup [x_1,x_0]).
	\end{multline}
\end{definition}
Then $E$ is defined and analytic in $\mathbb C \setminus ([x_3,x_2] \cup [x_1,x_0])$.
Indeed, because of the factor $(z-\beta_{11}\beta_{21})(z-\beta_{12}\beta_{22})$
the poles at $\beta_{11} \beta_{21}$ and $\beta_{12} \beta_{22}$ disappear.

We introduce two special real values $x^*$, $x^{**} < 0$ as follows 
\begin{align} \label{xstar} 
	x^* & = - \frac{\alpha_{12} \beta_{11} \beta_{22} \gamma_{21} + \alpha_{21}\beta_{12} \beta_{22} \gamma_{11} + \alpha_{12} \alpha_{21} \beta_{22} + \alpha_{12} \alpha_{22} \beta_{11}}{\beta_{12} \gamma_{11} \gamma_{22} + \beta_{22} \gamma_{11} \gamma_{21} + \alpha_{12} \gamma_{22} + \alpha_{22} \gamma_{11}}, \\
	\label{xstar2}
	x^{**} & = - \frac{\alpha_{11} \beta_{12} \beta_{21} \gamma_{22} + \alpha_{22}\beta_{11} \beta_{21} \gamma_{12} + \alpha_{11} \alpha_{21} \beta_{12} + \alpha_{11} \alpha_{22} \beta_{21}}{\beta_{11} \gamma_{12} \gamma_{21} + \beta_{21} \gamma_{12} \gamma_{22} + \alpha_{11} \gamma_{21} + \alpha_{21} \gamma_{12}}.  
\end{align}
At these points one of the off-diagonal entries of $\Phi$ vanishes, see Lemma \ref{lemma51}(e) below.
Associated to \eqref{xstar} and \eqref{xstar2} there are two special points $p^*$ and $p^{**}$ on the bounded real oval of the 
Riemann surface  given by,
\begin{equation} \label{pstar} 
	p^* = (x^*, \Phi_{22}(x^*)), \quad p^{**} = (x^{**}, \Phi_{11}(x^{**})),
\end{equation} 
see Figure~\ref{RSurface}.
We show in Lemma \ref{lemmaP} below that $p^*$ and $p^{**}$ are indeed on the bounded oval, 
and thus in particular 
\begin{equation} \label{xstarineq}
	 x_2 \leq x^* \leq x_1, \quad x_2 \leq x^{**} \leq x_1.
	 \end{equation}
\begin{lemma} \label{lemma51} The following hold,
	\begin{enumerate}
		\item[\rm (a)] $\Psi(z) := (z-\beta_{11} \beta_{12}) (z-\beta_{21} \beta_{22})  \Phi(z)$
		is a matrix-valued polynomial,
		\item[\rm (b)] $\Psi_{12}$ has degree one, and $x^*$ is a zero of $\Psi_{12}$, 
		\item[\rm (c)] $\Psi_{21}$ has degree two, with $0$ and $x^{**}$ being zeros of $\Psi_{21}$,
		\item[\rm (d)] $\Psi_{11}$ and $\Psi_{22}$ have degree two,
		\item[\rm (e)] $\Phi_{12}(x^*) = \Phi_{21}(x^{**}) = 0$.
	\end{enumerate}
\end{lemma}
\begin{proof}
	These are straightforward consequences of the definitions \eqref{phib1}--\eqref{defPhi},
	and the formulas \eqref{xstar} and \eqref{xstar2}.
\end{proof}

For $E$ defined in \eqref{defEz} we obtain the following. 
\begin{lemma} \label{lemma52}
	\begin{enumerate}
		\item[\rm (a)] $E_{11}(z) = E_{12}(z)$ and both are polynomials of degree $1$
		that vanish at $z=x^*$.
		\item[\rm (b)] 	$E_{21}(z) E_{22}(z)$ is a polynomial of degree three with zeros at $0$, $x^*$,
		and $x^{**}$.
	\end{enumerate}
\end{lemma}
\begin{proof}
	(a) is immediate from \eqref{defEz} and Lemma \ref{lemma51}(b,e).
	
	(b)
	Since $\lambda_1$ and $\lambda_2$ are the two eigenvalues of $\Phi$,
	we have $\lambda_1+ \lambda_2 = \Tr \Phi = \Phi_{11}+ \Phi_{22}$ and
	$\lambda_1 \lambda_2 = \det \Phi  = \Phi_{11} \Phi_{22} - \Phi_{21} \Phi_{12}$.
	Hence
	\begin{align*} (\lambda_1 - \Phi_{11})(\lambda_2-\Phi_{11}) & =
		(\Phi_{11} \Phi_{22} - \Phi_{21} \Phi_{12}) - (\Phi_{11}+ \Phi_{22}) \Phi_{11}
		+ \Phi_{11}^2 \\
		= - \Phi_{12} \Phi_{21},
		\end{align*}
	which is a rational function with at most double poles at $\beta_{11} \beta_{21}$ and $\beta_{12} \beta_{22}$.
	Using the notation $\Psi$ from Lemma \ref{lemma51}(a), we get from this and \eqref{defEz} that
	\[ E_{21}(z) E_{22}(z) = -\Psi_{12}(z) \Psi_{21}(z), \]
	which by Lemma \ref{lemma51}(b,c) is indeed a polynomial of degree three 
	with zeros at $0$, $x^*$ and $x^{**}$.	 
\end{proof}

According to Lemma \ref{lemma52}(b) we have that $x^*$ and $x^{**}$ are zeros of either $E_{21}$
or $E_{22}$. The following lemma gives more precise information. 
Recall that $p^*$ and $p^{**}$ are given in \eqref{pstar}.

\begin{lemma} \label{lemmaP}
	The following properties hold.
 \begin{enumerate}
		\item[\rm (a)] $p^*$ and $p^{**}$ are points on the bounded oval of the Riemann surface $\mathcal R$.
		\item[\rm (b)] If $p^*$ is on the second sheet, then $E_{21}(x^*) = 0$.
		\item[\rm (c)] If $p^*$ is on the first sheet, then $E_{22}(x^*) = 0$.
		
		\item[\rm (d)] If $p^{**}$ is on the first sheet, then $E_{21}(x^{**}) = 0$.
		\item[\rm (e)] If $p^{**}$ is on the second sheet, then $E_{22}(x^{**}) = 0$.

	\end{enumerate}
\end{lemma} 

\begin{proof}
	(a)
	 Since $\Phi_{12}(x^*) = 0$ by Lemma \ref{lemma51}(e), the matrix $\Phi(x^*)$ is
	 lower triangular. Its eigenvalues are thus on the diagonal, and therefore $(x^*, \Phi_{11}(x^*))$
	and $(x^*, \Phi_{22}(x^*))$ are real points on the Riemann surfaces.
	
	Hence they are either on the bounded or on the unbounded oval. To see that they
	are on the bounded oval, we use a continuity argument. If parameters are such that
	$x_0 = 0$ and $x_{3} = -\infty$, then the unbounded oval has no part
	on the negative real $z$-axis. This is e.g.\ the case for the models in \cite{BD23, DK21}.
	Since $-\infty < x^* < 0$ it must follow that $x^* \in [x_2, x_1]$ in that case.
	All notions vary continuously with the parameters, and in addition the parameter
	space is connected. Hence $x^* \in [x_2, x_1]$ for any set of positive parameters,
	and $p^* = (x^*, \Phi_{22}(x^*))$ is on the bounded oval.
	
	Similarly, $p^{**}$ is on the bounded oval, where we now use that 
	$\Phi_{21}(x^{**}) =0$, hence $\Phi_{21}(x^{**})$ is upper triangular, 
	and again we find that $\Phi_{11}(x^{**})$ and $\Phi_{22}(x^{**})$ are the eigenvalues
	of $\Phi(x^{**})$. 
	
	\medskip
	(b) The eigenvalues of $\Phi(x^*)$ are its diagonal elements
	$\Phi_{11}(x^*)$ and $\Phi_{22}(x^*)$. 
	If $p^*$ is on the second sheet then, by \eqref{pstar}, we have $\lambda_2(x^*) = \Phi_{22}(x^*)$,
	and consequently $\lambda_1(x^*) = \Phi_{11}(x^*)$. 
	Then $E_{21}(x^*) = 0$ follows from  \eqref{defEz}.
	
	\medskip
	(c)-(e) are proved similarly. 

\end{proof}
We continue to give the proof of Theorem \ref{theo21} under the following conditions that are generically satisfied, namely
\begin{itemize}
	\item We have 
	\begin{equation} \label{xstrictineq} 
		-\infty < x_3 < x_2 < x_1 < x_0 < 0, \end{equation}
	i.e., strict inequalities hold in \eqref{xineq}.
	\item  Also strict inequalities hold in \eqref{xstarineq}
	\begin{equation} \label{xstarstrict}
		x_2 < x^* < x_1, \quad \text{ and } \quad x_2 < x^{**} < x_1,
	\end{equation} and in addition
	\begin{equation} \label{xstardiff}
		x^* \neq x^{**}. \end{equation}
\end{itemize}
The proof can be easily adapted  to cover the other cases. Alternatively one may
use a limiting argument to prove Theorem \ref{theo21} for the cases where
one or more equalities in \eqref{xstrictineq}, \eqref{xstarstrict},
\eqref{xstardiff} hold, since all notions depend continuously on the
parameters. However we always assume $x_2 < x_1$ as for $x_2 = x_1$ the Riemann surface 
has genus zero.

 As a final preliminary, we collect information on the asymptotic 
 behavior of $\Phi(z)$ and $E(z)$ as $z \to \infty$. 
 
  \begin{lemma} \label{lemma54}
	\begin{enumerate}
		\item[\rm (a)] We have
		\begin{equation} \label{Phiinf} 
			\Phi(z) \to \Phi(\infty) := \begin{pmatrix} \gamma_{11}
				& 0 \\ \alpha_{11} & \gamma_{12} \end{pmatrix} 
				\begin{pmatrix} 1
					& 0 \\ \beta_{11} & 1 \end{pmatrix} 
					\begin{pmatrix} \gamma_{21}
						& 0 \\ \alpha_{21} & \gamma_{22} \end{pmatrix} 
						\begin{pmatrix} 1
							& 0 \\ \beta_{21} & 1 \end{pmatrix}, \end{equation}
       as $z \to \infty$.
		\item[\rm (b)] We have
		\begin{align} \lambda_1(z) & =  \label{lambda1atinf}  \begin{cases}
				\Phi_{11}(\infty) + O(z^{-1}), & \text{if } \Phi_{11}(\infty) > \Phi_{22}(\infty), \\
				\Phi_{22}(\infty) + O(z^{-1}), & \text{if } \Phi_{11}(\infty) < \Phi_{22}(\infty), 
				\end{cases}  \\
		 \lambda_2(z) & = \label{lambda2atinf} \begin{cases}
				\Phi_{22}(\infty) + O(z^{-1}), & \text{if } \Phi_{11}(\infty) > \Phi_{22}(\infty), \\
				\Phi_{11}(\infty) + O(z^{-1}), & \text{if } \Phi_{11}(\infty) < \Phi_{22}(\infty), 
				\end{cases} 
		\end{align}
		as $z \to \infty$.
		\item[\rm (c)] We have
		$E_{11}(z) = E_{12}(z) =  O(z)$ as $z \to \infty$, and
		\begin{align} \label{E21atinf}
				E_{21}(z) & = \begin{cases}
					O(z), & \text{if } \Phi_{11}(\infty) > \Phi_{22}(\infty), \\
					O(z^2), & \text{if } \Phi_{11}(\infty) < \Phi_{22}(\infty), 
				\end{cases} \\ \label{E22atinf}
				E_{22}(z) & = \begin{cases}
					O(z^2), & \text{if } \Phi_{11}(\infty) > \Phi_{22}(\infty), \\
					O(z), & \text{if } \Phi_{11}(\infty) < \Phi_{22}(\infty).
				\end{cases}
			\end{align}
	\end{enumerate}
\end{lemma} 

\begin{remark}
	We assumed that $x_3 > -\infty$, see \eqref{xstrictineq}  which means that
	there is no branching at infinity. This assumption amounts to $\gamma_{11} \gamma_{21} \neq 
	\gamma_{12} \gamma_{22}$, which in view of \eqref{Phiinf} means
	that $\Phi_{11}(\infty) \neq \Phi_{22}(\infty)$. This is the reason we do
	not consider $\Phi_{11}(\infty) = \Phi_{22}(\infty)$ in parts (b) and (c). In fact
	one would have fractional powers in the $O$-terms in this case, e.g.\ $E_{21}(z) = O(z^{3/2})$
	and $E_{22}(z) = O(z^{3/2})$ as $z \to \infty$. 
\end{remark}
\begin{proof}
(a) is immediate from \eqref{defPhi} and the formulas for $\phi^b$ and $\phi^g$
in \eqref{phib1}--\eqref{phig2}.

\medskip
(b) From item (a) we see that $\Phi(\infty)$ is a lower triangular matrix with eigenvalues
$\Phi_{11}(\infty) = \gamma_{11} \gamma_{21}$ and $ \Phi_{22}(\infty) = \gamma_{12} \gamma_{22}$, 
that are positive. Then $\lambda_1(z)$
tends to the larger of these two values, and $\lambda_2(z)$ to the smaller, as $z \to \infty$.	
Since $x_3 > -\infty$, we have $\Phi_{11}(\infty) \neq \Phi_{22}(\infty)$, 
and the $O(z^{-1})$ terms in \eqref{lambda1atinf} and \eqref{lambda2atinf} follow as there is no branching, 
and $\lambda_1$, $\lambda_2$ are analytic at infinity.

\medskip
(c) The statements about $E_{11}$ and $E_{12}$ are immediate from Lemma \ref{lemma52}(a),
while \eqref{E21atinf} and \eqref{E22atinf} follow from \eqref{defEz} and part (b).
\end{proof}

\subsubsection{Matrix-valued function $G$}

The matrix-valued function $G = \begin{pmatrix} g_1 & 0 \\ 0 & g_2 \end{pmatrix}$
will be a diagonal matrix with entries $g_1$, $g_2$ that are given in terms of
Jacobi theta functions.  

In order to construct $G$, we first have to observe that $\lambda$ and $z$ are meromorphic functions on $\mathcal R$ with two poles.
These are 
\begin{align} \label{polelambda} 
 	\text{poles of $\lambda$}: \quad & p_1 = (\beta_{11} \beta_{21}, \infty), \qquad  p_2 = (\beta_{12} \beta_{22}, \infty), \\ \label{polez}
 	\text{poles of $z$}: \quad & p_{\infty,1} = (\infty, \gamma_{11} \gamma_{21}), \quad p_{\infty,2} = (\infty, \gamma_{12} \gamma_{22}),  \end{align}
as can be checked from \eqref{phib1}--\eqref{phig2}.
Due to the ordering \eqref{Harnack1} the poles $p_1$ and $p_2$ are on the first
sheet of the Riemann surface. Generically they are distinct. The poles $p_{\infty,1}$
and $p_{\infty,2}$ correspond to the points at infinity of the Riemann surface.
Generically they are also distinct. If $\gamma_{11} \gamma_{21} > \gamma_{12} \gamma_{22}$
then $p_{\infty,1}$ is the point at infinity on the first sheet
and $p_{\infty,2}$ is the point at infinity on the second sheet. 
If $\gamma_{11} \gamma_{21} < \gamma_{12} \gamma_{22}$ then it is the other way around.
If $\gamma_{11} \gamma_{21} = \gamma_{12} \gamma_{22}$, then $x_3=-\infty$,
and infinity is a branch point of the Riemann surface.

The construction of $g_1, g_2$ depends on the four special points \eqref{polelambda}
and \eqref{polez}.  
	
\begin{definition} \label{def23}
	Let $p_1$, $p_2$, $p_{\infty,1}$, and $p_{\infty,2}$
	be as in \eqref{polelambda} and \eqref{polez}. 	
	Then the functions  
	$g_j : \mathbb C \setminus [x_3,x_0] \to \mathbb C$ for $j=1,2$ are defined as
	\begin{equation} \label{defgjz} 
		g_j(z) = \frac{ \vartheta(\mathcal A(z^{(j)})- \mathcal A(p_1) - K) \, \vartheta(\mathcal A(z^{(j)})- \mathcal A(p_2) - K)}
		{\vartheta(\mathcal A(z^{(j)})- \mathcal A(p_{\infty,1}) - K) \, \vartheta(\mathcal A(z^{(j)})- \mathcal A(p_{\infty,2}) - K)} 
	\end{equation}
	where $z^{(j)} = (z, \lambda_j(z))$ denotes the point on the $j$th sheet that projects 
	to $z$, and $\mathcal A$ is the Abel map \eqref{Abelmap}.
	Furthermore
	\begin{equation} \label{defG}
		G(z) = \begin{pmatrix} g_1(z) & 0 \\ 0 & g_2(z) \end{pmatrix}, \quad z \in \mathbb C \setminus [x_3,x_0].
		\end{equation}
\end{definition}

The ratio of products of Jacobi theta functions in \eqref{defgjz} 
gives a quasi-periodic function on the Riemann surface,
where we view $g_1$ as defined on the first sheet, and $g_2$ on the second sheet. 
They turn out to have a discontinuity on the bounded oval.  Define 
\begin{equation} \label{defomega} 
	\omega_0 = \mathcal A(p_1) + \mathcal A(p_2) \in (0,1). \end{equation}
Then it follows from \eqref{Jacobi2} that
\[ g_{1,+} = g_{1,-} e^{2 \pi i \omega_0}, \quad 
	g_{2,+} = g_{2,-} e^{-2\pi i \omega_0}, \qquad \text{on } \, [x_2,x_1], \]
 see also Lemma \ref{lemma56}.

\subsubsection{Matrix-valued function $E_{\omega}$ with $\omega \in \mathbb R$} \label{SubSectEomega}
For any $\omega \in \mathbb R$  we construct in the following the matrix-valued function $E_{\omega}$ used in Theorem \ref{theo21}. If $\omega \in \mathbb Z$ then $E_{\omega} = E$, and the formulas simplify.
The situation $N \omega_0 \in \mathbb Z$ with $\omega_0$ given by \eqref{defomega}
corresponds to the case of a torsion point in the paper of Borodin and Duits \cite{BD19}.

\begin{definition}	\label{def24}
		For $\omega \in \mathbb R$, we define
		\begin{multline} \label{defEomega}
			E_{\omega}(z)  = \\  \begin{pmatrix} E_{11}(z) \ds
				\frac{\vartheta(\mathcal A(z^{(1)}) - \omega -  \mathcal A_+(p^*)-K)}
				{\vartheta(\mathcal A(z^{(1)}) - \mathcal A_+(p^*) - K)} 		
				& E_{12}(z) \ds \frac{\vartheta(\mathcal A(z^{(2)}) - \omega -  \mathcal A_+(p^*)-K)}
				{\vartheta(\mathcal A(z^{(2)}) - \mathcal A_+(p^*) - K)} \\[10pt]
				E_{21}(z) \ds \frac{\vartheta(\mathcal A(z^{(1)}) - \omega -  \mathcal A_+(p^{**})-K)}
				{\vartheta(\mathcal A(z^{(1)}) - \mathcal A_+(p^{**}) - K)} 		
				& E_{22}(z) \ds \frac{\vartheta(\mathcal A(z^{(2)}) - \omega -  \mathcal A_+(p^{**})-K)}
				{\vartheta(\mathcal A(z^{(2)}) - \mathcal A_+(p^{**}) - K)} 
			 \end{pmatrix}, \\
		z \in \mathbb C \setminus [x_3,x_0],
		\end{multline}	
		with $p^*$ and $p^{**}$ as in \eqref{pstar} and $\mathcal A$ being the Abel map.
		As before, we use $z^{(j)}$, for $j=1,2$, to denote
		the point on the $j$th sheet of the Riemann surface that projects to $z$. 
	\end{definition}
Thus $E_{\omega}$ is the Hadamard pointwise product of $E$ with a matrix built out of
ratios of Jacobi theta functions. Its quasi-periodicity properties are listed in Lemma~\ref{lemma57} below

\subsection{Auxiliary lemmas}\label{Sect_Aux_Lem}
\subsubsection{Lemmas on boundary behavior}
The matrix-valued functions $E$, $G$ and $E_{\omega}$ 
that appear in the right-hand sides of the formulas \eqref{PNformula} and \eqref{PNhatformula}
are analytic in the upper and lower half planes. We study their boundary behaviors on the real line.
 
\begin{lemma} \label{lemma55}
	$E$ has boundary values $E_{\pm}$ on the real line satisfying
	\begin{align} \label{Ejump}
		E_+ = \begin{cases} E_-,  & \text{ on } \, (-\infty, x_3] \cup [x_2,x_1] \cup [x_0, \infty), \\
		E_- \sigma_1, & \text{ on } \, [x_{3}, x_2] \cup [x_1,x_0],
	\end{cases} \end{align}
	where $\sigma_1 = \begin{pmatrix} 0 & 1 \\ 1 & 0 \end{pmatrix}$.
\end{lemma}
\begin{proof}
	The jump  \eqref{Ejump} on $(-\infty, x_3] \cup [x_2,x_1] \cup [x_0, \infty)$
	is immediate as $E$ is analytic there. 
	If $x \in [x_{3}, x_2] \cup [x_1,x_0]$, then we are on the cuts
	of the Riemann surface where $\lambda_{1,\pm}(x) = \lambda_{2,\mp}(x)$.
	The jump \eqref{Ejump} on the cuts follows from this and the definition \eqref{defEz}.
\end{proof}

The boundary behavior of $G$ is described in the next lemma.
\begin{lemma} \label{lemma56}
	$G$ has boundary values $G_{\pm}$ on the real line 	satisfying
	\begin{equation} \label{Gjump}
		G_+ = \begin{cases}
			G_-, & \text{ on } \, (-\infty, x_3] \cup [x_0, \infty), \\
			\sigma_1 G_- \sigma_1, & \text{ on } \, [x_3, x_2] \cup [x_1,x_0], \\ 
			G_- e^{2\pi i \omega_0 \sigma_3},
			& \text{ on } \, [x_2,x_1],
		\end{cases}
	\end{equation} 
	where $\omega_0$ is given by \eqref{defomega} and
	 $\sigma_3 = \begin{pmatrix} 1 & 0 \\ 0 & -1 \end{pmatrix}$.
\end{lemma}
\begin{proof} 
	For $x \in \mathbb R$ we use the notation 	 
	\[ \mathcal A\left(x_{\pm}^{(j)}\right) = \lim_{z \to x, \, \pm \Im z > 0} \mathcal A\left(z^{(j)}\right). \]
	
	For $x \in (-\infty, x_3] \cup [x_0,\infty)$ we are on the unbounded
	oval where $\mathcal A$ takes values in $(-\frac{1}{2}, \frac{1}{2})$.
Given our definition of the Abel map	(see remark after \eqref{Abelmap}), there is no discontinuity there and we have
\begin{align*}
\mathcal A\big(x_{+}^{(j)}\big) = \mathcal A\left(x_{-}^{(j)}\right), \quad \text{for }  \, j = 1,2 \, \text{ and } \, x \in (-\infty, x_3] \cup [x_0,\infty),
\end{align*}
which by  definitions \eqref{defgjz} and \eqref{defG} means that
	$G_+ = G_-$ on $(-\infty, x_3] \cup [x_0,\infty)$,
	as claimed in \eqref{Gjump}.
	
	For $x \in [x_3,x_2] \cup [x_1,x_0]$ we are on one of the branch cuts
	and $x^{(1)}_{\pm} = x^{(2)}_{\mp}$ as points on the Riemann surface.
	The Abel map $\mathcal A$ has either  purely imaginary values at these points, 
	or takes limiting values in $\pm \frac{1}{2} + i \mathbb R$. In either case we have that
	$\mathcal A\big( x_\pm^{(1)} \big) - \mathcal A\big(x^{(2)}_\mp \big) \in \mathbb Z$.
	Because of the periodicity \eqref{Jacobi2} of the Jacobi theta function, 
	and the definition \eqref{defgjz} we obtain 
	$g_{1,\pm}(x) = g_{2,\mp}(x)$ which results in the jump \eqref{Gjump}
	on $[x_3, x_2] \cup [x_1,x_0]$.
		
	For $x \in [x_2, x_1]$ we are on the bounded oval and $x^{(j)}_{\pm} = x^{(j)}_{\mp}$
	for $j=1,2$. The Abel map $\mathcal A$ takes limiting values at these points
	in $[-\frac{1}{2}, \frac{1}{2}] \pm \frac{\tau}{2}$. Their difference is   
	\begin{align*} \mathcal A\big(x^{(1)}_+\big) - \mathcal A\left(x^{(1)}_-\right) & = -\tau, \\
		\mathcal A\big(x^{(2)}_+\big) - \mathcal A\left(x^{(2)}_-\right) & = \tau.
		\end{align*}
	We then use the quasi-periodicity property \eqref{Jacobi2} for all four Jacobi theta functions in \eqref{defgjz} to find
	\begin{equation} \label{g1g2jump}
		g_{1,+}(x)  = g_{1,-}(x) e^{2\pi i \omega_0}, \quad  
		g_{2,+}(x)  = g_{2,-}(x) e^{-2\pi i \omega_0},  \qquad \text{for }
		x \in [x_2,x_1],
		\end{equation}
	with 
	\begin{equation} \label{defomega2} 
		\omega_0 = \mathcal A\left(p_1\right) + \mathcal A\left(p_2\right) -
	 \mathcal A\left(p_{\infty,1}\right) - \mathcal A\left(p_{\infty,2}\right). \end{equation}
	This gives the jump \eqref{Gjump} for $G$ on $[x_2,x_1]$.
	To see that $\omega_0$ is also given by \eqref{defomega}, we observe that
	$\mathcal A(x^{(2)}) = - \mathcal A(x^{(1)})$ for every real $x > x_0$,
	which is an immediate consequence of the definition of the Abel map with base
	point $x_0$. Then $\mathcal A(p_{\infty,2}) = - \mathcal A(p_{\infty,1})$,
	and \eqref{defomega} follows from \eqref{defomega2}.
	\end{proof}
	
	We need more information about the boundary behavior of $E_{\omega}$.
	\begin{lemma} \label{lemma57} Let $\omega \in \mathbb R$.
		\begin{enumerate}
			\item[\rm (a)] 
			$E_{\omega}$ is well-defined and analytic in $\mathbb C \setminus [x_3,x_0]$.
			\item[\rm (b)] 
			$E_{\omega}$ has bounded boundary values on the real line.
			\item[\rm (c)] The boundary values satisfy 
			\begin{equation} \label{Eomegajump}
				E_{\omega,+} = \begin{cases}
					E_{\omega,-}, & \text{ on } \, (-\infty, x_3] \cup [x_0, \infty), \\
					E_{\omega,-}\sigma_1,  & \text{ on } \,[x_3, x_2]
					\cup [x_1,x_0], \\ 
					E_{\omega,-} e^{2\pi i \omega \sigma_3},
					& \text{ on } \, [x_2,x_1].
				\end{cases}
			\end{equation}
			\item[\rm (d)]
			If $p^*$ is on the second sheet, then
			\[ \left( E_{\omega} \right)_{j1}(z) = O(z-x^*) \quad \text{ as } z \to x^*, \quad \text{for } j=1,2.\]
			\item[\rm (e)]
			If $p^*$ is on the first sheet,	then
			\[ \left( E_{\omega}\right)_{j2}(z) = O(z-x^*) \quad \text{ as } z \to x^*, \quad \text{for } j=1,2.\]
		\end{enumerate}
	\end{lemma}
	\begin{proof}
		(a) Recall that $E_{\omega}$ is given by \eqref{defEomega}. It is the Hadamard product of $E$ with the matrix-valued function
		\begin{align} \label{defPsi}
			\Psi_{\omega}(z)  = \begin{pmatrix} \psi_j (\mathcal A( z^{(k)}); \omega) \end{pmatrix}_{j,k=1,2},  
		\end{align}	
		that is built out of the following ratios of Jacobi theta functions
		\begin{align} \label{psi1} 
			\psi_1(u;\omega) & = \frac{\vartheta(u - \omega -  \mathcal A_+(p^*)-K)}{\vartheta(u - \mathcal A_+(p^*) - K)}, \\
			\label{psi2} 
			\psi_2(u; \omega) & = \frac{\vartheta(u - \omega -  \mathcal A_+(p^{**})-K)}{\vartheta(u - \mathcal A_+(p^{**}) - K)}. 
		\end{align}	
		By \eqref{psi1} we have that  $\psi_1$ has simple poles at $\mathcal A_+(p^*) + (\mathbb Z + \tau \mathbb Z)$
		and $\psi_2$ has simple poles at $\mathcal A_+(p^{**}) + (\mathbb Z + \tau \mathbb Z)$.
		Since $p^*$ and $p^{**}$ are on the bounded oval, the denominators in \eqref{defEomega}
		do not vanish for $z \in \mathbb C \setminus [x_3,x_0]$. Therefore \eqref{defEomega} 
		is well defined and  analytic in $\mathbb C \setminus [x_3, x_0]$. 
		\medskip
		
		(b) The boundary values $E_{\omega,\pm}(x)$ exist for every $x \in \mathbb R$ 
		except possibly at the special
		values $x = x^*$ or $x = x^{**}$.
		If $z \to x^*$ then either $z^{(1)} \to p^*$ or $z^{(2)} \to p^*$, depending on whether
		$p^*$ is on the first or second sheet. Thus one of the denominators in the
		first row of \eqref{defEomega} tends to $0$ as $z \to x^*$. However, both $E_{11}$
		and $E_{12}$ have a zero at $x^*$, see Lemma \ref{lemma52}(a), 
		and since the Jacobi theta function $\vartheta$
		has only simple zeros, we find that the first row of \eqref{defEomega} remains bounded as $z \to x^*$. If $p^*$ is on the first sheet, then the zero of $E_{11}$ gets cancelled by the zero of the Jacobi theta function in the denominator, 
		while the zero of $E_{12}$ at $x^*$
		remains a zero of $E_{\omega}$, i.e.,  
		\begin{equation} \label{Eomegazero1}
			\left(E_{\omega,\pm} \right)_{12}(x^*) = 0,  \quad \text{ if $p^*$ is on the first sheet}.
		\end{equation}
		Similarly
		\begin{equation} \label{Eomegazero2} 
			\left(E_{\omega,\pm} \right)_{11}(x^*) = 0,  \quad \text{ if $p^*$ is on the second sheet}. 
		\end{equation}
	
		We encounter for $z \to x^{**}$ a similar situation in the second row of \eqref{defEomega},
		as one of the Jacobi theta functions in the denominator
		vanishes at $x^{**}$.  
		More concretely, if $p^{**}$ is on the first sheet then this happens in the $21$-entry, but
		then by Lemma \ref{lemmaP}(d) also $E_{21}(x^*) =0$.
		Similarly, if $p^{**}$ is on the second sheet, then this happens in the $22$-entry, 
		but we again have $E_{22}(x^*) =0$. In either case the pole is cancelled by a zero, 
		and $E_{\omega,\pm}$ also exists at $x^{**}$. This concludes the proof of (b).
		\medskip
		
		(c)	By the quasi-periodicity
		properties \eqref{Jacobi1} the functions $\psi_1$ and $\psi_2$ defined by  \eqref{psi1}, \eqref{psi2}
		satisfy
		\[ \psi_j(u+1; \omega) = \psi_j(u;\omega) \quad \text{and} \quad
		\psi_j(u+\tau; \omega) = e^{2\pi i \omega} \psi_j(u;\omega), \]
		for $j=1,2$.     Because of \eqref{Abeljump} and \eqref{defPsi} this implies
		\begin{equation} \label{JPsi1} 
			\Psi_{\omega,+}(x) =  \Psi_{\omega,-}(x) e^{2\pi i \omega \sigma_3}, \qquad x \in (x_2,x_1). 
		\end{equation}
				For $x \in (x_3, x_2) \cup (x_1,x_0)$ we have $x^{(1)}_{\pm} = x^{(2)}_{\mp}$ as
		points on the Riemann surface. The closed contour from $p_0$ to $x_{\pm}^{(1)}$ on the first sheet 
		and then from $x^{(2)}_{\mp}$ back to $p_0$ on the second sheet, is either
		contractible or homotopic to the $\bf{a}$ cycle or to $-\bf{a}$. Because of the
		normalization $\oint_{\bf a} \eta = 1$ of the holomorphic differential that is used
		in the Abel map \eqref{Abelmap} we find 
		$\mathcal A(x_{\pm}^{(1)}) - \mathcal A(x^{(2)}_{\mp}) \in \mathbb Z$, and
		thus by the first periodicity property \eqref{Jacobi2} and the definitions \eqref{defPsi}, \eqref{psi1}, \eqref{psi2},
		we obtain
		\begin{equation} \label{JPsi2} 
			\Psi_{\omega,+}(x) = \Psi_{\omega_-}(x) \sigma_1, \quad x \in  (x_{3}, x_2) \cup (x_1,x_0). 
		\end{equation}
		Combining \eqref{JPsi1}, \eqref{JPsi2} with the jump \eqref{Ejump} of $E$, and
		recalling the definition \eqref{defEomega},    we find \eqref{Eomegajump}. 
		\medskip
		
		(d) Suppose $p^*$ is on the second sheet. 
		Then $(E_{\omega,\pm})_{11}(x^*) = 0$ as already noted in \eqref{Eomegazero2}.
		By analyticity, we have $\left( E_{\omega}\right)_{11}(z) = O(z-x^*)$ as $z \to x^*$. 

		Also $E_{21}(z) = O(z-x^*)$ as $z \to x^*$ by Lemma \ref{lemmaP}(b) and analyticity,
		together with the assumption that $x^* \in (x_2, x_1)$, see \eqref{xstarstrict}.
		In \eqref{defEomega} $E_{21}$ is multiplied by  $\psi_2(\mathcal A(z^{(1)}))$
		which remains bounded as $z \to x^*$,
		since we assumed $x^{*} \neq x^{**}$, see \eqref{xstarineq}. 
		Thus $\left( E_{\omega}\right)_{21}(z) = O(z-x^*)$ as $z \to x^*$.
		 
		\medskip
		
		(e) The proof is similar to the proof of part (d). 
	\end{proof}	
	
\subsubsection{Lemmas on determinants}
	
The formulas \eqref{PNformula} and \eqref{PNhatformula} involve
the inverse matrices of $E$ and $E_{N\omega_0}$. To see that these
are well-defined we need to know that the inverses indeed exists.
This follows from the next lemma.

\begin{lemma} \label{lemma58} 
	Let $\omega \in \mathbb R$. Then
	\begin{equation} \label{detEomegasquare} 
		z \mapsto \left(\det E_{\omega}(z)\right)^2 \end{equation}
	is a degree six polynomial with a double zero at $z = x^*$ and
	simple zeros at the four branch points $x_j$, $j=0,1,2,3$, of the Riemann surface $\mathcal R$.
\end{lemma}

\begin{proof}
	From \eqref{Eomegajump} we see that $(\det E_{\omega})_+ = (\det E_{\omega})_-$
	on $(-\infty, x_{3}] \cup [x_2,x_1] \cup [x_0, \infty)$
	and $(\det E_{\omega})_+ = -(\det E_{\omega})_-$ on $[x_3, x_2] \cup [x_1,x_0]$,
	since $\det \sigma_1 = -1$ and $\det e^{2\pi i \omega \sigma_3} = 1$.
	Thus $\det E_{\omega}$ vanishes at $x_j$ for $j=0,1,2,3$.   
	It also implies that \eqref{detEomegasquare} has the same $\pm$ boundary values
	on the full real line, 	and \eqref{detEomegasquare} is therefore entire. 
	The formulas \eqref{defEz} and \eqref{defEomega} show that each entry
	of $E_{\omega}$ can have at most polynomial growth at infinity, and
	therefore \eqref{detEomegasquare} is a polynomial
	with zeros at the  branch points of the Riemann surface.
	
	There is also a double zero at $x^*$, as we show next. By Lemma \ref{lemma57}(d,e) 
	one of the columns of $E_{\omega}$ vanishes linearly at $x^*$. 
	Then $x^*$ is a zero of $z \mapsto \det E(z)$ and thus 
	a double zero (or possibly higher order zero) of \eqref{detEomegasquare}.
	
	We now have located six zeros of \eqref{detEomegasquare} 
	namely the branch points $x_j$, $j=0,1,2,3$ and 
	the double zero at $x^*$. We show that these are all the zeros by examining the behavior at infinity.
	From the fact that the Jacobi theta functions
	in \eqref{defEomega} remain bounded and bounded away from zero as $z \to \infty$,
	we first  conclude that 
	\[ \det E_{\omega}(z) =  O \left( E_{11}(z) E_{22}(z)\right)  +  O\left(E_{21}(z) E_{12}(z) \right) \]
	as $z \to \infty$. 	Then we use Lemma \ref{lemma54}(c) and we find
	$\det E_{\omega}(z) = O(z^3)$ as $z \to \infty$.  
	Thus the polynomial \eqref{detEomegasquare} has degree $\leq 6$, but then the degree is six since we already found six zeros.

	The lemma follows. 
\end{proof}

\begin{lemma} \label{lemma59}
	The functions $g_1$ and $g_2$ are analytic on $\mathbb C \setminus [x_3,x_0]$ with
	the following properties
	\begin{enumerate} 
		\item[\rm (a)]
		$g_1$ has simple zeros at $\beta_{11} \beta_{12}$ and $\beta_{21} \beta_{22}$ 
		(or a double zero in case these values are equal),
		and $g_1$, $g_2$ have no other zeros.
		\item[\rm (b)] There are non-zero constants $c_{\infty,j}$ for $j=1,2$ such that
		\begin{equation} \label{gjatinf} 
			g_j(z) = c_{\infty,j} z + O(1) \quad \text{ as } z \to \infty. \end{equation}
		\item[\rm (c)] $\det G = g_1 g_2$ is a degree two polynomial with zeros at $\beta_{11} \beta_{12}$ 
		and $\beta_{21} \beta_{22}$.
	\end{enumerate}
\end{lemma}
\begin{proof}
	(a) From its definition  \eqref{defgjz} we see that $g_j$ has a zero at $z$ if 
	and only if $\mathcal A(z^{(j)}) = \mathcal A(p_1)$
	or $\mathcal A(z^{(j)}) = \mathcal A(p_2)$. Since $p_1$ and $p_2$ are on the first sheet this
	happens if and only if $j=1$ and $z = \beta_{11} \beta_{12}$ or $z=\beta_{11} \beta_{12}$. This proves (a) as the zeros of the Jacobi theta function are simple.
	\medskip
	
	(b)
	The poles of \eqref{defgjz} arise from the zeros of one of the factors in the denominator,
	and this happens when  $\mathcal A(z^{(j)}) = \mathcal A(p_{\infty,1})$
	or $\mathcal A(z^{(j)}) = \mathcal A(p_{\infty,2})$. Thus there are poles at the
	two infinities and part (b) follows.
	\medskip
	
	(c) From \eqref{Gjump} we easily get $ \det G_+ = \det G_-$ on the full real line.
	Hence $\det G = g_1 g_2$ has an analytic continuation to an entire function.
	Because of (b) it is a polynomial of degree two, and the zeros come from
	the zeros of $g_1$ as given in part (a) of the lemma. 
\end{proof}

\subsection{Proof of Theorem \ref{theo21}}
\label{SectProofThm}
\begin{proof}
After all the preparations we are ready for the proof of Theorem \ref{theo21}.
Let $P_N$ and $\widehat{P}_N$ be defined by the right hand sides of \eqref{PNformula} and \eqref{PNhatformula}. We show that for an appropriate choice of invertible
matrices $C_N$ and $\widehat{C}_N$, the functions $P_N$ and $\widehat{P}_N$ are monic matrix-valued
polynomials of degree $N$ such that $P_NW$ and $W \widehat{P}_N$ are also matrix-valued polynomials. We give the details for $P_N$, since the proofs for $\widehat{P}_N$
follow in the same way. 

The right-hand side of \eqref{PNformula} is well-defined and analytic for $z \in \mathbb C \setminus \mathbb R$ because of Lemma \ref{lemma57}(a) and Lemma \ref{lemma58}.

The boundary values on the real line are finite, except possibly at $x^*$
or at one of the branch points, since there the inverse matrix $E^{-1}$ has 
a singularity due to Lemma \ref{lemma58}. 
We calculate the jumps of $P_N$ using \eqref{Ejump}, \eqref{Gjump}, \eqref{Eomegajump}.  
The three  factors are
analytic across $(-\infty, x_3) \cup (x_0,\infty)$.  On $(x_3,x_2) \cup (x_1,x_0)$
we find using $\sigma_1^2 = I_2$,
\begin{align*}
	\left(E_{-N\omega_0} G^N E^{-1} \right)_+
	& = \left( E_{-N \omega_0,-} \sigma_1 \right) \left( \sigma_1  G_-^N \sigma_1 \right)  (E_- \sigma_1)^{-1} \\  
	& = \left(E_{-N\omega_0} G^N E^{-1} \right)_-.
	\end{align*}
For $x \in (x_2, x_1)$, $x \neq x^*$, we find by 
	 \eqref{Ejump}, \eqref{Gjump}, \eqref{Eomegajump},
\begin{align*}
 \left(E_{-N\omega_0} G^N E^{-1} \right)_+
	& = \left( E_{-N \omega_0} \right)_- e^{-2\pi i N \omega_0 \sigma_3}  G_-^N e^{2\pi i N \omega_0 \sigma_3}  E_-^{-1} \\  
	& = \left(E_{-N\omega_0} G^N E^{-1} \right)_-.
\end{align*}
In the second step we used  that $G^N$ commutes with 
$e^{2\pi i N \omega_0 \sigma_3}$, as both are diagonal matrices.

It follows that $E_{-N\omega_0} G^N E^{-1}$ has a continuation to a meromorphic
function on  $\mathbb C$ with isolated singularities at $x^*$, and $x_j$ for $j=0,1,2,3$.
By Lemma \ref{lemma58}, $z \mapsto \det E(z)$ vanishes like $(z-x)^{1/2}$ as $z \to x$ if 
$x$ is one of the branch points. Then
$\left(E_{-N\omega_0} G^N E^{-1} \right)(z) =  O((z-x)^{-1/2})$, which is enough to conclude that
the isolated singularities at the branch points are removable.
As $z \to x^*$ we have $(\det E(z))^{-1} = O((z-x^*)^{-1})$ by Lemma \ref{lemma58}, 
which would allow for a simple
pole at $z=x^*$. But as noted in Lemma \eqref{lemma57}(d,e), $E$ and $E_{\pm N\omega_0}$ 
have a zero column in common at $z=x^*$. Suppose it is the first column (i.e., $p^*$ is
on the second sheet). Then as $z \to x^*$,
\begin{align*} 
	E_{-N \omega_0}(z) & = \begin{pmatrix} O(z-x^*) & O(1) \\ O(z-x^*) & O(1) \end{pmatrix}, \\ 
	E^{-1}(z) & = \frac{1}{\det E(z)} \begin{pmatrix} E_{22}(z) & -E_{21}(z) \\
		-E_{12}(z) & E_{11}(z) \end{pmatrix}	= \begin{pmatrix} O((z-x^*)^{-1}) & O((z-x^*)^{-1}) \\
			O(1) & O(1) \end{pmatrix}. \end{align*}
Since $G^N$ is a diagonal matrix that remains bounded near $x^*$, we obtain from the above
\begin{multline*} 
	  E_{-N\omega_0}(z) G^N(z) E^{-1}(z) \\ = 
		\begin{pmatrix} O(z-x^*) & O(1) \\ O(z-x^*) & O(1) \end{pmatrix} 
		\begin{pmatrix} O(1) & 0 \\ 0 & O(1) \end{pmatrix}
		 \begin{pmatrix} O((z-x^*)^{-1}) & O((z-x^*)^{-1}) \\ O(1) & O(1) \end{pmatrix} \\
		 = \begin{pmatrix} O(1) & O(1) \\ O(1) & O(1) \end{pmatrix} \quad \text{ as } z \to x^*.
	\end{multline*} 
A similar calculation works if the second column is the zero column for $E(x^*)$.
This shows that  $ E_{-N\omega_0} G^N E^{-1}$ remains bounded near $x^*$.
Hence the singularity at $x^*$ is removable and $ E_{-N\omega_0} G^N E^{-1}$ is  
a matrix-valued polynomial. 

To determine the degree we examine the behavior at infinity. Suppose $\Phi_{11}(\infty) > \Phi_{22}(\infty)$. Then by \eqref{E21atinf} and \eqref{E22atinf},
\begin{equation} \label{Easymp} 
	E(z) = \begin{pmatrix} O(z) & O(z) \\ O(z) & O(z^2) \end{pmatrix} \end{equation}
and
\begin{equation} \label{Einvasymp} 
	E^{-1}(z) = \frac{1}{\det E(z)} \begin{pmatrix} E_{22}(z) & -E_{12}(z) \\
		- E_{21}(z) & E_{11}(z) \end{pmatrix} =  O(z^{-3})
			\begin{pmatrix} O(z^2) & O(z) \\ O(z) & O(z) \end{pmatrix} \end{equation}
as $z \to \infty$, where we used Lemma \ref{lemma58}.
By \eqref{defEomega}, $E_{-N\omega_0}(z)$ has the same behavior as $E(z)$ as $z \to \infty$,
since the Jacobi theta functions in \eqref{defEomega} have finite, non-zero limits at infinity. Hence by \eqref{Easymp}
\begin{equation} \label{Eomegaasymp} 
	E_{-N\omega_0}(z) = \begin{pmatrix} O(z) & O(z) \\ O(z) & O(z^2) \end{pmatrix} 
	\quad \text{ as } z \to \infty. \end{equation}
Furthermore, by \eqref{defG} and Lemma \ref{lemma59}(b)
\begin{equation} \label{Gasymp}
	G(z) = \begin{pmatrix} O(z) & 0 \\ 0 & O(z) \end{pmatrix} \quad \text{ as } z \to \infty. 
\end{equation} 	
Combining \eqref{Einvasymp}, \eqref{Eomegaasymp} and \eqref{Gasymp} we find
\begin{multline*} 
	E_{-N\omega_0}(z) G^N(z) E^{-1}(z) \\
	 = \begin{pmatrix} O(z) & O(z) \\
		O(z) & O(z^2) \end{pmatrix}  \begin{pmatrix} O(z^N) & 0 \\ 0 & O(z^N) \end{pmatrix} 
	 \begin{pmatrix} O(z^{-1}) & O(z^{-2}) \\ O(z^{-2}) & O(z^{-2}) \end{pmatrix} \\
	 = \begin{pmatrix} O(z^N) & O(z^{N-1}) \\ O(z^N) & O(z^N) \end{pmatrix}
	 \quad \text{ as } z \to \infty
\end{multline*}
which shows that $E_{-N\omega_0} G^N E^{-1}$ is a matrix-valued polynomial of degree $\leq N$
in case $\Phi_{11}(\infty) > \Phi_{22}(\infty)$, with a coefficient of $z^N$ that is lower
triangular.  The same conclusion follows in the other case
$\Phi_{11}(\infty) < \Phi_{22}(\infty)$.

Next note that
\[  \det \left( E_{-N \omega_0} G^N E^{-1} \right)
	= \frac{\det E_{-N \omega_0}}{\det E}  \left( \det G \right)^N. \]
Because of Lemma \ref{lemma58}, $\frac{\det E_{-N \omega_0}}{\det E(\omega_0)}$
is a finite non-zero constant, and 
$\det G = c_{\infty,1} c_{\infty,2} (z-\beta_{11} \beta_{12}) (z-\beta_{21} \beta_{22})$
by Lemma \ref{lemma59}(b,c). 
We conclude that for some non-zero constant $c_N$,
\[ \det \left( E_{-N \omega_0}(z) G^N(z) E^{-1}(z) \right)  = c_N (z-\beta_{11} \beta_{12})^N (z-\beta_{21} \beta_{22})^N.\]
Therefore $E_{-N \omega_0} G^N E^{-1}$ has  exactly degree $N$ with an invertible leading
coefficient.  Define $C_N$ to be the inverse of this leading coefficient, then 
$P_N$ given by \eqref{PNformula} is indeed a monic matrix-valued polynomial of degree $N$. As the leading coefficient of $E_{-N\omega_0}G^N E^{-1}$
is a lower triangular matrix, its inverse $C_N$ is lower triangular as well.

To complete the proof  we show that $P_NW$ is a matrix-valued polynomial as well.
Since $W$ is a rational matrix function with poles at 
$\beta_{11} \beta_{12}$ and $\beta_{21} \beta_{22}$ only, it is enough to show
that these poles are cancelled in the product $P_N W$.
The columns of $E$ are the eigenvectors of $\Phi$ and $W = \Phi^N$. We thus have
$W = E \begin{pmatrix} \lambda_1^N & 0 \\ 0 & \lambda_2^N \end{pmatrix} E^{-1}$.
Then by  \eqref{PNformula}
\begin{equation} \label{PNW} 
	P_N W = C_N E_{-N \omega_0} \begin{pmatrix} (g_1 \lambda_1)^N & 0 \\ 0 & (g_2 \lambda_2)^N \end{pmatrix}
E^{-1}.  \end{equation}
Recall that $\lambda_1$ has simple poles at $\beta_{11} \beta_{12}$ and $\beta_{21} \beta_{22}$,
see Definition \ref{def23} (or a double pole in case these values are equal).
The poles of $\lambda_1$ are cancelled by the zeros of $g_1$, see Lemma \ref{lemma59}(a).
Furthermore, $\lambda_2$ has no poles, and also the other factors in the right-hand side of \eqref{PNW}
remain bounded at the poles of $W$. Thus $P_NW$ has no poles. From its behavior at infinity (see \eqref{Phiinf}), it follows that $P_NW$ is a matrix-valued polynomial.
Thus $P_N$ given by the right-hand side of \eqref{PNformula} is the monic
matrix-valued polynomial of degree $N$ for the weight $W$. As mentioned before, the claim for $\widehat P_N$ follows in a similar manner. This completes
the proof of Theorem \ref{theo21}.    
\end{proof}

\subsection*{Acknowledgements}
We are grateful to Tomas Berggren for interesting discussions.
Both authors are  supported by Methusalem grant METH/21/03 – long term structural funding of the 
Flemish Government.

\end{document}